\documentclass{article}
\usepackage{amsmath}
\usepackage{amsfonts}
\usepackage{amsthm}
\usepackage{enumerate}

\newtheorem{thm}{Theorem}[section]
\newtheorem{prop}[thm]{Proposition}
\newtheorem{cor}[thm]{Corollary}
\newtheorem{lem}[thm]{Lemma}

\theoremstyle{remark}
\newtheorem{rmk}[thm]{Remark}

\theoremstyle{definition}
\newtheorem{defn}[thm]{Definition}

\begin{document}

\title{Normalizers of Congruence Groups in $SL_{2}(\mathbb{R})$ and Automorphisms of Lattices}

\author{Shaul Zemel}

\maketitle

\section*{Introduction}

If $\Gamma$ is a congruence subgroup of $SL_{2}(\mathbb{Z})$ then there are several reasons to study the normalizer of $\Gamma$ in $SL_{2}(\mathbb{R})$. The main motivation is, except pure interest, that the quotient of that normalizer modulo $\Gamma$ embeds (under some assumptions on elliptic points) into the group of automorphisms of the Riemann surface $X_{\Gamma}$ associated with $\Gamma$. Moreover, elements of this normalizer act on spaces of modular forms with respect to $\Gamma$, endowing these already rich spaces with additional structure.

For the most classical congruence groups, namely $\Gamma_{0}(N)$ for natural $N$, many references describe this normalizer: See, e.g., \cite{[LN]} or Section 3 of \cite{[CN]}. We also mention the earlier work \cite{[N1]}, which shows that the normalizer of that group in $SL_{2}(\mathbb{Z})$ is $\Gamma_{0}\big(\frac{N}{\sigma}\big)$ where $\sigma$ is the largest divisor of 24 whose square divides $N$. Further references investigate the quotient of that normalizer modulo $\Gamma_{0}(N)$. This quotient is a very simple group if the powers of 2 and 3 which divide $N$ are very small, but otherwise it gets significantly more complicated (see \cite{[AL]} for the first results on that quotient, though \cite{[AS]} and \cite{[B]} later corrected some errors in that reference). The aforementioned quotient embeds naturally into the automorphism group of the associated modular curve $X_{0}(N)$. \cite{[KM]} shows, together with the complementary works \cite{[E]} and \cite{[H]}, that the full automorphism group of $X_{0}(N)$ arises in this way, apart for the three cases in which $N$ is 37, 63, or 108. For these three values of $N$ the normalizer of $\Gamma_{0}(N)$ produces a subgroup of index 2 inside the automorphism group. We also mention \cite{[N2]} for some general results on normalizers of congruence subgroups of $SL_{t}(\mathbb{Z})$ inside $SL_{t}(\mathbb{R})$ for any $t\geq2$.

A tool which many of these references use is the Big Picture $\Omega$, first defined in \cite{[C]}, which is a certain graph whose vertices are the finitely generated subgroups of full rank in $\mathbb{Q}^{2}$ modulo homothety, with edges according to an explicit rule. \cite{[L1]} uses it in order to determine the normalizer of the image of $\Gamma_{1}(N)$ inside $PSL_{2}(\mathbb{R})$, and it also appears in the construction of the algorithm, developed in \cite{[L3]}, for determining normalizers of general subgroups (after one finds generators for the subgroup). We mention that \cite{[L2]} is concerned with the normalizers of groups which are slightly larger than $\Gamma_{0}(N)$, and not contained in $SL_{2}(\mathbb{Z})$, and uses it for finding normalizers of certain subgroups of the Hecke groups $G_{4}$ and $G_{6}$. On the other hand, \cite{[KK]} independently obtained the results of \cite{[L1]} using more elementary tools, similar to the ones used in the current paper, and related them to the automorphism group of $X_{1}(N)$ (the latter connection is easy since $\Gamma_{1}(N)$ has no elliptic points for $N\geq5$). The case of the principal congruence subgroups $\Gamma(N)$ was considered in \cite{[BKX]}, which showed, among other things, that the normalizer of $\Gamma(N)$ is $SL_{2}(\mathbb{Z})$ for every $N$, and that if the modular curve $X(N)$ has genus at least 2, then the quotient $SL_{2}(\mathbb{Z}/N\mathbb{Z})$ (divided by $\{\pm I\}$) is its automorphism group. The authors of the latter paper claim, however, that these results were already well-known before.

\smallskip

The aim of this paper is to find the normalizers of various families of congruence groups, which are much more general than just $\Gamma_{0}(N)$ and $\Gamma_{1}(N)$. The main groups we investigate lie between these two families of groups, so that any such group is associated to a unique subgroup $H$ of $\Gamma_{0}(N)/\Gamma_{1}(N)\cong(\mathbb{Z}/N\mathbb{Z})^{\times}$. We begin by introducing a group, which we denote by $\Gamma_{0}^{*,s_{N}}(N)$, containing $\Gamma_{0}(N)$ with finite index, in which all these normalizers are contained, and then give conditions on elements of $\Gamma_{0}^{*,s_{N}}(N)$ which are equivalent to normalizing the subgroup which is associated with $H$. Using these conditions we write the normalizer explicitly for two types of subgroups $H$, namely the kernel of the projection to $(\mathbb{Z}/D\mathbb{Z})^{\times}$ (i.e., the normalizers of the intersection $\Gamma_{0}(N)\cap\Gamma_{1}(D)$) for some divisor $D$ of $N$ and the $m$-torsion subgroups for divisors $m$ of the exponent $\lambda(N)$ of $(\mathbb{Z}/N\mathbb{Z})^{\times}$, together with some additional results.

After the preprint was ready, I was informed about the existence of the paper \cite{[IJK]}, which independently obtained a criterion for normalizing which is similar to ours (compare Theorem 2.1 of that reference with our Corollary \ref{inGammasigK}). However, the goals of that reference are different from ours: While we aim only for determining the normalizers, with no restriction on $N$, \cite{[IJK]} concentrates on square-free $N$ (where finding the normalizers becomes much easier) but focuses on the structure of associated quotient group.

In addition, some lattices of signature $(2,1)$ have discriminant kernels which are (isomorphic to) congruence subgroups---see, e.g., \cite{[LZ]}, or \cite{[BO]} and the references therein. The lattices appearing in these references are related to $\Gamma_{1}(N)$ and $\Gamma_{0}(N)$ respectively, and they are also part of a larger family of lattices $L(N,D)$, whose discriminant kernels turn out to be congruence subgroups as well. The automorphism group of such a lattice is contained in the normalizer of its discriminant kernel, a simple observation which links the two questions to one another. We define these lattices $L(N,D)$, and show how the tools developed for determining normalizers can also be used for finding the automorphism group of $L(N,D)$ and its discriminant kernel.

We remark again that in this paper we study only the normalizers themselves, not the structure of the quotient modulo the congruence group. This is so, since this quotient is complicated in general: See the case of $\Gamma_{0}(N)$ considered in \cite{[AS]} and \cite{[B]}, or the case of $\Gamma_{1}(N)$, where Corollary \ref{Gamma1N0N} shows that this group is an extension of $\{\pm1\}^{\{p|N\}}$ by $(\mathbb{Z}/N\mathbb{Z})^{\times}$. As the action of the former group on the latter is, in general, non-trivial (it is described explicitly in Proposition \ref{ALact}), and the extension is non-trivial as well, we leave the questions about the structure of the quotient for further research.

\smallskip

This paper is divided into 4 sections. In Section \ref{MatGp} we introduce the group $\Gamma_{0}^{*,s_{N}}(N)$ and certain subgroups of it, and prove some of their properties. In Section \ref{DetNorm} we establish the tool for determining the normalizer of any intermediate group between $\Gamma_{1}(N)$ and $\Gamma_{0}(N)$. Section \ref{CongSub} describes the normalizers of several families of subgroups, in particular $\Gamma_{0}(N)\cap\Gamma_{1}(D)$ for $D|N$ and the subgroups associates with $m$-torsion in $(\mathbb{Z}/N\mathbb{Z})^{\times}$. Finally, Section \ref{Lat} presents the lattices $L(N,D)$ and calculates their automorphism groups as well as their discriminant kernels.

\smallskip

I would like to thank D. Jeol for sharing the details of \cite{[IJK]} with me, and therefore exposed me to some additional literature that I was not aware of at the time. Many thanks are due to the anonymous referee, for the numerous suggestions and correction that improved greatly the clarity and smoothness of the presentation of this paper.

\section{Some Matrix Groups \label{MatGp}}

In this Section we present some types of groups, which will appear as normalizers of congruence subgroups below. We recall that $\Gamma_{0}(N)$, where $N$ is any positive integer, is the group consisting of those matrices $\big(\begin{smallmatrix} a & b \\ c & d\end{smallmatrix}\big) \in SL_{2}(\mathbb{Z})$ such that $N|c$, and that $\Gamma_{1}(N)$ is the subgroup of $\Gamma_{0}(N)$ in elements of which the diagonal entries $a$ and $d$ are congruent to 1 modulo $N$. Adding the condition $N|b$ yields the principal congruence subgroup $\Gamma(N)$, which is normal in $SL_{2}(\mathbb{Z})$ with quotient $SL_{2}(\mathbb{Z}/N\mathbb{Z})$.

Given any positive integer $M$, we define $s_{M}$ to be the square root of its square part, and $t_{M}$ is the ``remainder''. This means that $s_{M}$ and $t_{M}$ are the unique positive integers such that $M=s_{M}^{2}t_{M}$ with $t_{M}$ square-free. In addition, for a prime number $p$ and an integer $M$, we denote the maximal integer $k$ such that $p^{k}|M$ by $v_{p}(M)$ (the \emph{$p$-adic valuation} of $M$). A divisor $\mu$ of an integer $N$ is called \emph{exact} if it is co-prime to $\frac{N}{\mu}$.

\smallskip

\begin{defn}
Let $\Gamma_{0}^{*,s_{N}}(N)$ be the set of matrices $A \in SL_{2}(\mathbb{R})$ which admit a presentation of the form $\big(\begin{smallmatrix} a\sqrt{\mu} & b/\sqrt{\mu} \\ c\frac{N}{\mu}\sqrt{\mu} & d\sqrt{\mu}\end{smallmatrix}\big)$, where $\mu$ is an exact divisor of $N$, $a$ and $d$ are in $\frac{1}{s_{\mu}}\mathbb{Z}$, and $b$ and $c$ are in $\frac{1}{s_{N/\mu}}\mathbb{Z}$. The equality $ad\mu-bc\frac{N}{\mu}=1$, which is equivalent to $A$ indeed being in $SL_{2}(\mathbb{R})$, will be referred to as the \emph{$SL_{2}$ condition} in what follows. \label{Gamma0*sNdef}
\end{defn}

The properties of $\Gamma_{0}^{*,s_{N}}(N)$ which will be of interest to us are the following ones.
\begin{prop}
\begin{enumerate}[$(i)$]
\item $\Gamma_{0}^{*,s_{N}}(N)$ is a subgroup of $SL_{2}(\mathbb{R})$.
\item Conjugation by $\Gamma_{0}^{*,s_{N}}(N)$ takes $\Gamma_{1}(N)$ into $\Gamma_{0}(N)$.
\end{enumerate} \label{1stgrp}
\end{prop}

\begin{proof}
We first observe that $\Gamma_{0}^{*,s_{N}}(N)$ is stable under inversion. For evaluating the product of the matrix appearing in Definition \ref{Gamma0*sNdef} with another such matrix, say $\big(\begin{smallmatrix} e\sqrt{\nu} & f/\sqrt{\nu} \\ g\frac{N}{\nu}\sqrt{\nu} & h\sqrt{\nu}\end{smallmatrix}\big)$, we define $\delta=\gcd\{\mu,\nu\}$ and $\kappa=\frac{\mu\nu}{\delta^{2}}$. Then $\kappa$ is an exact divisor of $N$, and the product in question equals
\begin{equation}
\begin{pmatrix} (ae\delta+bg\frac{N}{\delta\kappa})\sqrt{\kappa} & (af\frac{\mu}{\delta}+bh\frac{\nu}{\delta})/\sqrt{\kappa} \\ (ce\frac{\nu}{\delta}+dg\frac{\mu}{\delta})\frac{N}{\kappa}\sqrt{\kappa} & (cf\frac{N}{\delta\kappa}+dh\delta)\sqrt{\kappa}\end{pmatrix}. \label{prodform}
\end{equation}
As $s_{\mu}a$, $s_{\mu}d$, $s_{N/\mu}b$, $s_{N/\mu}c$, $s_{\nu}e$, $s_{\nu}h$, $s_{N/\nu}f$, and $s_{N/\nu}g$ are integers, it is easy to verify that multiplying the expressions appearing in parentheses in the diagonal entries of the matrix from Equation \eqref{prodform} by $s_{\kappa}$ and the ones appearing in the off-diagonal entries by $s_{N/\kappa}$ yield integral values as well. Hence the product also satisfies the conditions of Definition \ref{Gamma0*sNdef}, establishing part $(i)$.

For part $(ii)$ we observe that conjugating any matrix $\gamma=\big(\begin{smallmatrix} e & f \\ g & h\end{smallmatrix}\big) \in M_{2}(\mathbb{R})$ by the matrix $A$ from Definition \ref{Gamma0*sNdef} yields
\begin{equation}
\begin{pmatrix} bdg-acfN+ade\mu-bch\frac{N}{\mu} & a^{2}f\mu-ab(e-h)-b^{2}\frac{g}{\mu} \\ d^{2}g\mu+cdN(e-h)-c^{2}f\frac{N^{2}}{\mu} & acfN-bdg+adh\mu-bce\frac{N}{\mu}\end{pmatrix}. \label{conjform}
\end{equation}
Hence if $\gamma\in\Gamma_{1}(N)$ then the conditions on $a$, $b$, $c$, and $d$ in Definition \ref{Gamma0*sNdef} combine with the fact that $\frac{g}{N}$, $\frac{e-h}{N}$, and $f$ are integers to show that the conjugated matrix $A\gamma A^{-1}$ lies in $\Gamma_{0}(N)$. This proves part $(ii)$, which completes the proof of the proposition.
\end{proof}

Regarding the uniqueness and normalization of the presentation from Definition \ref{Gamma0*sNdef}, we obtain the following corollary.
\begin{cor}
Let $A$ be an element of $\Gamma_{0}^{*,s_{N}}(N)$, presented as in Definition \ref{Gamma0*sNdef}, and let $p$ be a prime dividing $N$.
\begin{enumerate}[$(i)$]
\item If $v_{p}(N)$ is odd (i.e., $p|t_{N}$) then $p$ divides either $ad\mu$ or $bc\frac{N}{\mu}$ (according to whether $p$ divides $\mu$ or $\frac{N}{\mu}$).
\item If $v_{p}(N)$ is even then one may obtain a different presentation by considering the exact divisor of $N$ that differs from $\mu$ by the factor $p^{v_{p}(N)}$.
\end{enumerate} \label{uniqpres}
\end{cor}

\begin{proof}
If $v_{p}(N)$ is odd and $p|\mu$ (resp. $p\big|\frac{N}{\mu}$) then the powers of $p$ in the denominators of $a$ and $d$ (resp. $b$ and $c$) can cancel at most a power of $2v_{p}(s_{N})=v_{p}(N)-1$ from the expression $p^{v_{p}(N)}$ appearing in $\mu$ (resp. $p\big|\frac{N}{\mu}$). This proves part $(i)$. For part $(ii)$, note that if $\mu=p^{2v_{p}(s_{N})}\nu$ then the matrix $A$ from Definition \ref{Gamma0*sNdef} can also be written as $\big(\begin{smallmatrix} ap^{v_{p}(s_{N})}\sqrt{\nu} & (b/p^{v_{p}(s_{N})})/\sqrt{\nu} \\ (c/p^{v_{p}(s_{N})})\frac{N}{\nu}\sqrt{\nu} & dp^{v_{p}(s_{N})}\sqrt{\nu}\end{smallmatrix}\big)$, while in case $p$ divides $\frac{N}{\mu}$ to an even power $2v_{p}(s_{N})$ then it can take the form $\big(\begin{smallmatrix} (a/p^{v_{p}(s_{N})})\sqrt{\nu} & bp^{v_{p}(s_{N})}/\sqrt{\nu} \\ cp^{v_{p}(s_{N})}\frac{N}{\nu}\sqrt{\nu} & (d/p^{v_{p}(s_{N})})\sqrt{\nu}\end{smallmatrix}\big)$. As the new rational coordinates have the required bounded denominators, this proves the corollary.
\end{proof}

Of the several presentations an element of $\Gamma_{0}^{*,s_{N}}(N)$ can have according to \ref{uniqpres}, some presentations are more convenient than others.
\begin{lem}
\begin{enumerate}[$(i)$]
\item Let $\gamma$ be an element of $\Gamma_{0}^{*,s_{N}}(N)$ presented as in Definition \ref{Gamma0*sNdef}. The two summands from the $SL_{2}$ condition, as well as the four products $ab$, $ac$, $bd$, and $cd$, are independent of the presentation.
\item Any element of $\Gamma_{0}^{*,s_{N}}(N)$ has at least one presentation in which the four products from part $(i)$ involve no cancelations.
\end{enumerate} \label{nocanN}
\end{lem}

\begin{proof}
Part $(i)$ is clear from the explicit formulae for the presentation changes in the proof of Corollary \ref{uniqpres} (or from the fact that the asserted expressions are, up to multiplication by $N$, products of two of the entries of the matrix itself). Now, the conditions from Definition \ref{Gamma0*sNdef} show that only cancelations of primes $p$ dividing $N$ have to be considered. Assuming that $p|\mu$, cancelation in powers of $p$ may only occur if $p$ divides the numerators of either $b$ or $c$. But then the number $bc\frac{N}{\mu}$ would be divisible by $p$, so that $ad\mu$ will be prime to $p$ (since their difference is 1). This can only happen if $v_{p}(N)$ is even and both $a$ and $d$ have the full power $p^{v_{p}(s_{N})}$ in their denominators. But then applying Corollary \ref{uniqpres} to use the divisor $\nu=\mu/p^{v_{p}(N)}$ would give a form in which the numbers in the off-diagonal entries may have $p$-powers in their denominators but the numerators of the new numbers in the diagonal entries are not divisible by $p$. The case where $p\big|\frac{N}{\mu}$ is established by the same argument, interchanging the roles of the diagonal and off-diagonal elements, and using the divisor $p^{v_{p}(N)}\mu$ instead. This establishes part $(ii)$ and completes the proof of the lemma.
\end{proof}

\begin{rmk}
The proof of Lemma \ref{nocanN} shows that a presentation involves cancelations in a prime $p|N$ only in case $p|ad\mu$ but instead of dividing $\mu$ it divides $\frac{N}{\mu}$, or the other way around (i.e., $bc\frac{N}{\mu}$ and $\mu$ are divisible by $p$, and not $\frac{N}{\mu}$). Part $(ii)$ of Lemma \ref{nocanN} allows us to avoid such presentations. A presentation as in part $(ii)$ of Lemma \ref{nocanN} is also in line with the fact that for primes dividing $t_{N}$, the summand from Corollary \ref{uniqpres} (i.e., $ad\mu$ or $bc\frac{N}{\mu}$) corresponding to the divisor of $N$ that is divisible by $p$ is also divisible by $p$. \label{canpdiv}
\end{rmk}

\smallskip

We shall make use of another simple lemma.
\begin{lem}
\begin{enumerate}[$(i)$]
\item The operation sending $\mu$ and $\nu$ of $N$ to $\kappa$ in the proof of Proposition \ref{1stgrp} defines a group structure on the set of exact divisors of $N$, making it a group which is isomorphic to $\{\pm1\}^{\{p|N\}}$.
\item For any divisor $D|N$, there exists a canonical projection from $\{\pm1\}^{\{p|N\}}$ to $\{\pm1\}^{\{p|D\}}$, which is surjective and its kernel consists of all the exact divisors $\mu$ of $N$ which are co-prime to $D$.
\end{enumerate} \label{pm1pdivN}
\end{lem}

\begin{proof}
If $\mu=\prod_{p|\mu}p^{v_{p}(N)}$ and $\nu=\prod_{p|\nu}p^{v_{p}(N)}$ then $\kappa$ is the product of $p^{v_{p}(N)}$ over all the primes dividing $\mu$ or $\nu$ but not both. This proves part $(i)$. For part $(ii)$ the map taking the component of $p|D$ to itself and the component of $p|N$ which does not divide $D$ to the trivial element is clearly well-defined and surjective, and the translation to divisors of $N$ immediately yields the desired assertion. This proves the lemma.
\end{proof}

A more explicit description for the projection map from part $(ii)$ of Lemma \ref{pm1pdivN} can be given by the formula $\mu\mapsto\gcd\{\mu,D\}$. The resulting group is based on exact divisors of $D$, a set which also produces a group isomorphic to $\{\pm1\}^{\{p|D\}}$ by part $(i)$ of that lemma. However, for our purposes it would be more convenient to consider elements $\{\pm1\}^{\{p|D\}}$, when arising as a quotient of $\{\pm1\}^{\{p|N\}}$, as classes of exact divisors of $N$ modulo the appropriate equivalence relation (identifying divisors that differ only by powers of primes dividing $N$ but not $D$) rather than as divisors of $D$.

We now generalize Definition \ref{Gamma0*sNdef} as follows.
\begin{defn}
Let $\sigma$ be any divisor of $s_{N}$. Then we define $\Gamma_{0}^{*,\sigma}(N)$ to be the set of elements of $\Gamma_{0}^{*,s_{N}}(N)$ for which the power of $p$ dividing any of the denominators of $a$, $b$, $c$, and $d$ in a presentation as in Definition \ref{Gamma0*sNdef} which satisfies the conditions of Lemma \ref{nocanN} is at most $v_{p}(\sigma)$. In particular, the group $\Gamma_{0}^{*}(N)=\Gamma_{0}^{*,1}(N)$, in elements of which $a$, $b$, $c$, and $d$ are integers, is the group obtained from $\Gamma_{0}(N)$ by adding the \emph{Atkin--Lehner involutions}. \label{Gamma0*sigmadef}
\end{defn}

\begin{rmk}
Lemma \ref{nocanN} shows that the condition of Definition \ref{Gamma0*sigmadef} is satisfied if and only if $\sigma ab$ and $\sigma cd$, or equivalently $\sigma ac$ and $\sigma bd$, are integral. Part $(i)$ of Lemma \ref{uniqpres} shows that these equivalent characterizations of elements of $\Gamma_{0}^{*,\sigma}(N)$ are satisfied also in presentations that may not satisfy the conditions of Lemma \ref{nocanN}. \label{Gamma*sigmaeq}
\end{rmk}

\begin{prop}
\begin{enumerate}[$(i)$]
\item For any $\sigma|s_{N}$ the set $\Gamma_{0}^{*,\sigma}(N)$ from Definition \ref{Gamma0*sigmadef} is a subgroup of $\Gamma_{0}^{*,s_{N}}(N)$ that contains $\Gamma_{0}(N)$.
\item Two different presentations of elements of $\Gamma_{0}^{*,\sigma}(N)$, both of which satisfy the conditions of Lemma \ref{nocanN}, may arise from the operations considered in part $(ii)$ of Corollary \ref{uniqpres}, but only using primes $p|N$ that do not divide $\frac{N}{\sigma^{2}}$.
\item The index $[\Gamma_{0}^{*,\sigma}(N):\Gamma_{0}(N)]$ is $\sigma^{2}\prod_{p|\sigma,\ v_{p}(N)=2v_{p}(\sigma)}\big(1+\frac{1}{p}\big)\cdot\prod_{p|N/\sigma^{2}}2$.
\end{enumerate} \label{Gamma0*sigma}
\end{prop}

\begin{proof}
It is clear that $\Gamma_{0}^{*,\sigma}(N)$ is closed under inversion. Consider now the formula for the product of two elements appearing in Equation \eqref{prodform}, and assume that the two factors lie in $\Gamma_{0}^{*,\sigma}(N)$ and are presented as in Lemma \ref{nocanN}. Then the same argument as in the proof of Proposition \ref{1stgrp}, but with any number $s_{M}$ replaced by $\gcd\{s_{M},\sigma\}$, establishes part $(i)$ since the containment $\Gamma_{0}(N)\subseteq\Gamma_{0}^{*,\sigma}(N)$ is obvious. Alternatively, one may prove this part by considering the expressions from Remark \ref{Gamma*sigmaeq} in the formula for the product appearing in Equation \eqref{prodform}. Part $(ii)$ follows from Remark \ref{canpdiv}, since the only case where a prime can divide $\mu$ but not $ad\mu$, or $\frac{N}{\mu}$ but not $bc\frac{N}{\mu}$, in an element of $\Gamma_{0}^{*,\sigma}(N)$ is when $2v_{p}\big(\gcd\{s_{\mu},\sigma\}\big)=v_{p}(N)$ or $2v_{p}\big(\gcd\{s_{N/\mu},\sigma\}\big)=v_{p}(N)$, i.e., when $2v_{p}(\sigma)=v_{p}(N)$.

For part $(iii)$, first note that part $(ii)$ shows that the map taking an element $A\in\Gamma_{0}^{*,\sigma}(N)$ to the divisor $\mu$ of $N$ appearing in a presentation satisfying the conditions of Lemma \ref{nocanN} is well-defined up to at most the kernel of the map from $\{\pm1\}^{\{p|N\}}$ to $\{\pm1\}^{\{p|N/\sigma^{2}\}}$ described in part $(ii)$ of Lemma \ref{pm1pdivN}. By parts $(i)$ and $(ii)$ of that lemma, the formula for the product in Equation \eqref{prodform} shows that $\Gamma_{0}^{*,\sigma}(N)$ admits a well-defined group homomorphism to $\{\pm1\}^{\{p|N/\sigma^{2}\}}$, which is clearly surjective. Moreover, the kernel of this map consists precisely of those matrices which admit a presentation as in Lemma \ref{nocanN} with $\mu=1$. Now, the formula from Definition \ref{Gamma0*sNdef} and the condition from Definition \ref{Gamma0*sigmadef} show that the matrix $\frac{1}{\sqrt{\sigma}}\big(\begin{smallmatrix} \sigma & 0 \\ 0 & 1\end{smallmatrix}\big)$ conjugates this kernel to $\Gamma_{0}\big(\frac{N}{\sigma^{2}}\big)$, while its subgroup $\Gamma_{0}(N)$ is taken by this operation to the group $\Gamma_{0}^{0}\big(\frac{N}{\sigma},\sigma\big)$ consisting of those elements $\big(\begin{smallmatrix} a & b \\ c & d\end{smallmatrix}\big)\in\Gamma_{0}\big(\frac{N}{\sigma}\big)$ in which $\sigma|b$. We thus have to compare the indices of these congruence subgroups in $SL_{2}(\mathbb{Z})$. It is known (see, e.g., Section 1.2 of \cite{[DS]}) that the index any group of the form $\Gamma_{0}(M)$ in $SL_{2}(\mathbb{Z})$ is $M\prod_{p|M}\big(1+\frac{1}{p}\big)$, and working modulo the principal congruence subgroup $\Gamma(M)$ we see that a subgroup of the form $\Gamma_{0}^{0}(M,D)$ with $D|M$ has index $D$ in $\Gamma_{0}(M)$. Hence the index of our conjugate of $\Gamma_{0}(N)$ is the same index $N\prod_{p|N}\big(1+\frac{1}{p}\big)$ as $\Gamma_{0}(N)$ itself in $SL_{2}(\mathbb{Z})$ (recall that $\sigma|s_{N}$, so that the prime divisors of $\frac{N}{\sigma}$ coincide with those of $N$), while the conjugate of our kernel has index $\frac{N}{\sigma^{2}}\prod_{p|N/\sigma^{2}}\big(1+\frac{1}{p}\big)$. Taking the quotient between these indices, noting that the sets of primes differ precisely by the ones appearing above, and multiplying by the cardinality of $\{\pm1\}^{\{p|N/\sigma^{2}\}}$ (to go from the index of $\Gamma_{0}(N)$ in the kernel to the index in $\Gamma_{0}^{*,\sigma}(N)$ itself), yield the asserted value of the index. This completes the proof of the proposition.
\end{proof}

\begin{rmk}
The proof of part $(iii)$ in Proposition \ref{Gamma0*sigma} used the conjugation of the kernel appearing there by the matrix $\frac{1}{\sqrt{\sigma}}\big(\begin{smallmatrix} \sigma & 0 \\ 0 & 1\end{smallmatrix}\big)$. Extending the conjugation to all of  $\Gamma_{0}^{*,\sigma}(N)$ yields the Atkin--Lehner group $\Gamma_{0}^{*}\big(\frac{N}{\sigma^{2}}\big)$, with no fractions. Moreover, the map from $\Gamma_{0}^{*,\sigma}(N)$ to the quotient $\{\pm1\}^{\{p|N/\sigma^{2}\}}$ of $\{\pm1\}^{\{p|N\}}$ as in part $(ii)$ of Lemma \ref{pm1pdivN} and the one from $\Gamma_{0}^{*}\big(\frac{N}{\sigma^{2}}\big)$ to $\{\pm1\}^{\{p|N/\sigma^{2}\}}$ based on exact divisors of $\frac{N}{\sigma^{2}}$ commute with this conjugation. This is the reason for the notation used by \cite{[CN]} and others for $\Gamma_{0}^{*,\sigma}(N)$, for $\sigma$ (denoted by $h$ in that reference) being the value appearing in Corollary \ref{Gamma1N0N}. However, we shall stick to our $\Gamma_{0}^{*,\sigma}(N)$, since these groups, rather than their conjugates $\Gamma_{0}^{*}\big(\frac{N}{\sigma^{2}}\big)$, are the ones appearing as normalizers below. \label{conjGamma*N/sigma2}
\end{rmk}

\section{Determining Normalizers \label{DetNorm}}

The intermediate groups between $\Gamma_{1}(N)$ and $\Gamma_{0}(N)$ are in correspondence with the subgroups of the quotient group $(\mathbb{Z}/N\mathbb{Z})^{\times}$. We denote $\Gamma_{H}$ the intermediate group which corresponds to the subgroup $H\subseteq(\mathbb{Z}/N\mathbb{Z})^{\times}$. In this Section we establish criteria for determining the normalizer of $\Gamma_{H}$ in general. This results will be used for describing the normalizer of $\Gamma_{H}$ for particular types of subgroups $H$ explicitly in the following section.

\smallskip

We begin by considering elements $A \in SL_{2}(\mathbb{R})$ such that the corresponding conjugate $A\Gamma_{1}(N)A^{-1}$ of $\Gamma_{1}(N)$ is contained in $\Gamma_{0}(N)$. It is clear that if $A$ normalizes some $\Gamma_{H}$ then it has this property. Now, part $(ii)$ of Proposition \ref{1stgrp} shows that elements of $\Gamma_{0}^{*,s_{N}}(N)$ do this. We would like to show that they are the only ones. The first step is given in the following lemma.
\begin{lem}
Assume that conjugation by a matrix $A=\big(\begin{smallmatrix} e & f \\ g & h\end{smallmatrix}\big) \in SL_{2}(\mathbb{R})$ takes $\Gamma_{1}(N)$ into $\Gamma_{0}(N)$. Then the expressions $e^{2}$, $eg$, $\frac{g^{2}}{N}$, $2Nef$, $N(eh+fg)$, $2gh$, $Nf^{2}$, $Nfh$, and $h^{2}$ are all integral. \label{intent}
\end{lem}

\begin{proof}
We consider only the three elements $\big(\begin{smallmatrix} 1 & 1 \\ 0 & 1\end{smallmatrix}\big)$, $\big(\begin{smallmatrix} 1 & 0 \\ N & 1\end{smallmatrix}\big)$, and $\big(\begin{smallmatrix} 1+N & -N \\ N & 1-N\end{smallmatrix}\big)$ of $\Gamma_{1}(N)$. Conjugation by $A$ yields the matrices $\big(\begin{smallmatrix} 1-eg & e^{2} \\ -g^{2} & 1+eg\end{smallmatrix}\big)$, $\big(\begin{smallmatrix} 1+Nfh & -Nf^{2} \\ Nh^{2} & 1-Nfh\end{smallmatrix}\big)$, and $\big(\begin{smallmatrix} 1+N(eh+fg+fh-eg) & N(e^{2}-2ef-f^{2}) \\ N(h^{2}+2gh-g^{2}) & 1-N(eh+fg+fh-eg)\end{smallmatrix}\big)$ respectively. Now, the first (resp. second) conjugated element lies in $\Gamma_{0}(N)$ if and only if the first (resp. last) three asserted numbers are integral. Assuming this, we find by subtracting the corresponding multiples of these numbers that the third conjugated element is in $\Gamma_{0}(N)$ precisely when the remaining three numbers are in $\mathbb{Z}$. This proves the lemma.
\end{proof}

We shall also make use of the following lemma.
\begin{lem}
Let $\alpha$, $\beta$, $\pi$, $\rho$, $\kappa$, and $N$ be integers, such that $\alpha$ and $\beta$ are positive and $\gcd\{\alpha,\beta\}$, $\gcd\{\alpha,\rho\}$, and $\gcd\{\beta,\pi\}$ all equal to 1. Assume that the quotients $\frac{\pi^{2}\alpha\kappa}{N\beta}$, $\frac{\rho^{2}\beta\kappa}{N\alpha}$, $\frac{\rho}{\alpha}\big(\frac{\kappa\pi\rho}{N}-1\big)$, $\frac{2\pi}{\beta}\big(\frac{\kappa\pi\rho}{N}-1\big)$, and $\frac{1}{\alpha\beta\kappa}\big(\frac{\kappa\pi\rho}{N}-1\big)^{2}$ are integers. Then $\alpha=\beta=1$, and $\kappa$ is a divisor of $N$. \label{alphabeta1}
\end{lem}

\begin{proof}
The integrality of the first two quotients, together with the co-primality conditions, imply that $\alpha\beta|\kappa$. Multiplying the third quotient by $N$, we find that $\alpha|N$ as well. If $N$ is even then multiplying the fourth quotient by $\frac{N}{2}$ yields the same assertion for $\beta$. If $N$ is odd then $\beta|2N$, and we claim that $\beta$ is odd. Indeed, multiplying the last quotient by $N^{2}$ produces a quotient whose numerator is odd if $\beta$ is even, and $\beta$ (and even $\beta^{2}$) appears in the denominator. Therefore $\alpha\beta|N$ as well. But then canceling $\alpha\beta$ from each expression involving $\frac{\kappa}{N}$ puts us back in the initial situation. Therefore $\alpha\beta$ divides both $\kappa$ and $N$ infinitely many times, so that it must be equal to 1. Substituting this and multiplying the last quotient by $N$, we find that $\kappa$ divides $N+\frac{(\kappa\pi\rho)^{2}}{N}$. Let $p$ be a prime, and assume that $v_{p}(\kappa)>v_{p}(N)$. But then $v_{p}\big(\frac{(\kappa\pi\rho)^{2}}{N}\big)>v_{p}(\kappa)>v_{p}(N)$, so that by adding $N$ we obtain a number whose $p$-adic valuation equals precisely $v_{p}(N)$. But such a number cannot be divisible by $\kappa$ if $v_{p}(\kappa)>v_{p}(N)$. This contradiction shows that $v_{p}(\kappa) \leq v_{p}(N)$ for every prime $p$, which amounts to $\kappa$ dividing $N$. This proves the lemma.
\end{proof}

We can now prove the desired assertion.
\begin{prop}
If conjugation by an element $A \in SL_{2}(\mathbb{R})$, written as in Lemma \ref{intent}, sends $\Gamma_{1}(N)$ into $\Gamma_{0}(N)$, then $A\in\Gamma_{0}^{*,s_{N}}(N)$. \label{ambgrp}
\end{prop}

\begin{proof}
Applying Lemma \ref{intent}, we can use the integrality of the expressions appearing in that lemma. If $g=0$ then the fact that $e^{2}$ and $h^{2}$ are integers and $eh=1$ (by the $SL_{2}$ condition) implies that $e=h=\pm1$. We have $f\in\frac{1}{N}\mathbb{Z}$ (since $Nfh\in\mathbb{Z}$), and the fact that $Nf^{2}$ is also integral implies that the denominator of the reduced form of $b=f$ must divide $s_{N}$. The assertion thus holds if $g=0$, with $\mu=1$.

We therefore assume that $g\neq0$. Using the integrality of $\frac{g^{2}}{N}\in\mathbb{Z}$ from Lemma \ref{intent}, and then of $eg$ and $2gh$, we find that
\begin{equation}
g=\pm\sqrt{Nt}\quad\mathrm{for\ some\ }t\in\mathbb{N},\quad\mathrm{as\ well \ as}\quad e=\frac{p}{\sqrt{Nt}}\quad\mathrm{and}\quad h=\frac{q}{2\sqrt{Nt}} \label{gehtpq}
\end{equation}
with integers $p$ and $q$. Since $h^{2}\in\mathbb{Z}$ as well, we obtain $4Nt|q^{2}$, so that $2|q$ and we can write $h=\frac{r}{\sqrt{Nt}}$ with $r\in\mathbb{Z}$.

The analysis will be easier if we separate divisors. Let $\delta=\gcd\{t,p,r\}>0$. Then the numbers $\alpha=\gcd\big\{\frac{t}{\delta},\frac{p}{\delta}\big\}$ and $\beta=\gcd\big\{\frac{t}{\delta},\frac{r}{\delta}\big\}$ are positive, co-prime, and both divide $\frac{t}{\delta}$. Hence the latter number is divisible by their product. We can thus write \[p=\delta\alpha\pi,\quad r=\delta\beta\rho,\quad\mathrm{and}\quad t=\delta\alpha\beta\tau,\quad\mathrm{with}\quad\tau\in\mathbb{N},\quad\pi\in\mathbb{Z},\quad\mathrm{and}\quad\rho\in\mathbb{Z},\] satisfying the co-primality conditions
\begin{equation}
\gcd\{\tau,\pi\rho\}=1,\quad\gcd\{\alpha,\beta\rho\}=1,\quad\mathrm{and}\quad\gcd\{\beta,\alpha\pi\}=1. \label{coprim}
\end{equation}
Substituting these values of $p$, $r$, and $t$, as well as $f=\frac{eh-1}{g}$ from the $SL_{2}$ condition, transforms and extends Equation \eqref{gehtpq} to
\begin{equation}
e=\frac{\pi\sqrt{\delta\alpha}}{\sqrt{N\beta\tau}},\quad g=\pm\sqrt{N\delta\alpha\beta\tau},\quad h=\frac{\rho\sqrt{\delta\beta}}{\sqrt{N\alpha\tau}},\quad\mathrm{and}\quad f=\pm\frac{\frac{\delta\pi\rho}{N\tau}-1}{\sqrt{N\delta\alpha\beta\tau}}, \label{eghf}
\end{equation}
where the two $\pm$ signs are the same.

Now, the numerator of $f$ must be integral (e.g., by the integrality of $Nf^{2}$ from Lemma \ref{intent}), so that the first co-primality condition in Equation \eqref{coprim} implies that $\tau|\delta$. We therefore write $\delta=\kappa\tau$ with $\kappa\in\mathbb{N}$, and observing that the values of $f$ and $g$ in Equation \eqref{eghf} involve the expression $\pm\sqrt{\delta\tau}=\pm\sqrt{\kappa}\cdot\tau$, we can remove the assumption $\tau>0$ and absorb the sign into $\tau$. Now, the numbers from Lemma \ref{alphabeta1} are $e^{2}$, $h^{2}$, $Nfh\tau$, $2Nef\tau$, and $Nf^{2}\tau^{2}$, which are all integral by Lemma \ref{intent}. Applying Lemma \ref{alphabeta1}, we get $\alpha=\beta=1$ and write $N$ as $\kappa\nu$, and Equation \eqref{eghf} takes the form
\begin{equation}
e=\frac{\pi}{\sqrt{\nu}},\quad g=\tau\kappa\sqrt{\nu},\quad h=\frac{\rho}{\sqrt{\nu}},\quad\mathrm{and}\quad f=\frac{\frac{\pi\rho}{\nu}-1}{\tau\kappa\sqrt{\nu}}. \label{withnu}
\end{equation}

We now use the integrality of $e^{2}$ and $h^{2}$ from Equation \eqref{withnu} to deduce that $\nu$ divides the square of $\gcd\{\pi,\rho\}$. Let $J$ be the set of primes $p$ diving $N$ which still divide $\frac{\gcd\{\pi,\rho\}^{2}}{\nu}$. The prime divisors of $\kappa$ are not in $J$, since the integrality of $\kappa\nu b^{2}\tau^{2}$ from Lemma \ref{intent} (recall the decomposition of $N$) implies that $\kappa$ divides the square of $\frac{\pi\rho}{\nu}-1$, and the latter number is congruent to $-1$ modulo any prime lying in $J$. Therefore the divisor $\mu=\prod_{p \in J}p^{v_{p}(N)}$ divides $\nu$, and the quotient is a square $\omega^{2}$ since $v_{p}(\nu)=2v_{p}\big(\gcd\{\pi,\rho\}\big)$ for any prime divisor $p$ of $N$ not lying in $J$. The divisor $\mu$ is clearly exact, and by substituting $\nu=\omega^{2}\mu$ and $\kappa=\frac{N}{\omega^{2}\mu}$ in Equation \eqref{withnu} we find that $A$ has the form from Definition \ref{Gamma0*sNdef} with the rational numbers $a=\frac{\pi}{\omega\mu}$, $c=\frac{\tau}{\omega}$, $d=\frac{\rho}{\omega\mu}$, and $b=\frac{\omega\mu}{N\tau}\big(\frac{\pi\rho}{\nu}-1\big)$. The expressions involving squares which must be integral by Lemma \ref{intent} are $a^{2}\mu$, $d^{2}\mu$, $c^{2}\frac{N}{\mu}$, and $b^{2}\frac{N}{\mu}$, which shows that $a$, $b$, $c$, and $d$ must satisfy the conditions from Definition \ref{Gamma0*sNdef}. This completes the proof of the proposition.
\end{proof}

\smallskip

The main technical tool for determining normalizers is the following refinement of Propositions \ref{ambgrp} and \ref{1stgrp}.
\begin{lem}
Let $H$ be a subgroup of $(\mathbb{Z}/N\mathbb{Z})^{\times}$, take some $A\in\Gamma_{0}^{*,s_{N}}(N)$, and present it as in Definition \ref{Gamma0*sNdef}. Then $A$ normalizes $\Gamma_{H}$ if and only if the following conditions are satisfied:
\begin{enumerate}[$(i)$]
\item The denominators of $ab$ and $cd$ divide any difference $e-h$ for an element $\gamma=\big(\begin{smallmatrix} e & f \\ g & h\end{smallmatrix}\big)$ in $\Gamma_{H}$.
\item $H$ contains the kernel of the projection onto $(\mathbb{Z}/K\mathbb{Z})^{\times}$, where $\frac{N}{K}$ is the least common multiple of the denominators of $ac$ and $bd$.
\item For any $\gamma$ as in condition $(i)$, adding $\pm bc\frac{N}{\mu}(e-h)$ to the diagonal elements gives again elements of $H$.
\end{enumerate} \label{detnorm}
\end{lem}

\begin{rmk}
Condition $(ii)$ in Lemma \ref{detnorm} is equivalent to $H$ being invariant under additive translations by multiples of $K$. Indeed, the fact that $\frac{N}{K}\big|s_{N}$ shows that the primes dividing $N$ already divide $K$, so that such translations do not affect co-primality to $N$. As such a translation takes an element of $(\mathbb{Z}/N\mathbb{Z})^{\times}$ to its image under multiplication by an element of the kernel of the projection from that condition, this indeed proves the claim. \label{transK}
\end{rmk}

\begin{proof}
Consider the matrix from Equation \eqref{conjform}, in which we assume that the matrix $\gamma=\big(\begin{smallmatrix} e & f \\ g & h\end{smallmatrix}\big)$ lies in $\Gamma_{H}$. Condition $(i)$ is equivalent to the upper right entry there being integral, and to the lower left entry there being in $N\mathbb{Z}$. The interpretation of condition $(ii)$ given in Remark \ref{transK} means, as the proof of Proposition \ref{1stgrp} shows, that conjugation by $A$ takes $\Gamma_{1}(N)$ into $\Gamma_{H}$. When these two conditions are satisfied, then the diagonal entries of the matrix from Equation \eqref{conjform} differ from that of $A$ by $\pm bc\frac{N}{\mu}(e-h)$, up to expressions which are dealt with in condition $(ii)$. Hence condition $(iii)$ is equivalent, under the other two conditions, to normalizing $\Gamma_{H}$. This proves the lemma.
\end{proof}

\smallskip

Using the groups $\Gamma_{0}^{*,\sigma}(N)$ from Definition \ref{Gamma0*sigmadef}, we can rephrase Lemma \ref{detnorm} in the following way.
\begin{cor}
For a subgroup $H$ of $(\mathbb{Z}/N\mathbb{Z})^{\times}$ denote by $K_{H}$ the minimal multiple of $s_{N}t_{N}$ such that the kernel of the projection from $(\mathbb{Z}/N\mathbb{Z})^{\times}$ to $(\mathbb{Z}/K_{H}\mathbb{Z})^{\times}$ is contained in $H$. Let $\sigma_{H}=\gcd\big\{\frac{N}{K_{H}},\eta_{H}\big\}$, where $\eta_{H}$ is the $\gcd$ of all the differences $e-h$ where $e$ and $h$ are integers that map to elements of $H$ that are mutual inverses. Then the normalizer of $\Gamma_{H}$ is contained in $\Gamma_{0}^{*,\sigma_{H}}(N)$. More precisely, this normalizer consists of those elements $\big(\begin{smallmatrix} a\sqrt{\mu} & b/\sqrt{\mu} \\ c\frac{N}{\mu}\sqrt{\mu} & d\sqrt{\mu}\end{smallmatrix}\big)$ of $\Gamma_{0}^{*,\sigma_{H}}(N)$ such that if $e$ and $h$ are represent elements of $H$ that are inverse to each other then the residues of $ade\mu-bch\frac{N}{\mu}$ and $adh\mu-bce\frac{N}{\mu}$ modulo $N$ are also elements of $H$ with that property. \label{inGammasigK}
\end{cor}

\begin{proof}
We first observe that $K_{H}$ is a well-defined divisor of $N$ (since $H$ contains the trivial subgroup of $(\mathbb{Z}/N\mathbb{Z})^{\times}$), and that for divisors $K$ of $N$ which are divisible by $s_{N}t_{N}$ the kernel of the map to $(\mathbb{Z}/K\mathbb{Z})^{\times}$ determines $K$. This is easily seen through the fact that $\frac{\varphi(N)}{\varphi(K)}=\frac{N}{K}$ for such divisors. Since we consider only elements of $\Gamma_{0}^{*,s_{N}}(N)$ (so that any $\sigma$ is a divisor of $s_{N}$), condition $(ii)$ of Lemma \ref{detnorm} is satisfied precisely for elements which lie in $\Gamma_{0}^{*,N/K_{H}}(N)$. It is now clear that $\eta_{H}|N$ (since $\big(\begin{smallmatrix} 1+N & 1 \\ N & 1\end{smallmatrix}\big)\in\Gamma_{1}(N)$), and that condition $(i)$ of Lemma \ref{detnorm} is satisfied for an element of $\Gamma_{0}^{*,N/K_{H}}(N)$ if and only if that element is in $\Gamma_{0}^{*,\sigma_{H}}(N)$. Now, the product of $ade\mu-bch\frac{N}{\mu}$ and $adh\mu-bce\frac{N}{\mu}$ is $eh-(e-h)^{2}ad\mu \cdot bc\frac{N}{\mu}$, which is 1 plus a multiple of $N$ plus a multiple of $N\frac{(e-h)^{2}}{\sigma_{H}^{2}}$ if $\big(\begin{smallmatrix} a\sqrt{\mu} & b/\sqrt{\mu} \\ c\frac{N}{\mu}\sqrt{\mu} & d\sqrt{\mu}\end{smallmatrix}\big)\in\Gamma_{0}^{*,\sigma_{H}}(N)$ by Remark \ref{Gamma*sigmaeq}. As $\sigma_{H}|e-h$, the asserted elements are indeed inverses. Since the terms that do not involve $ad\mu$ or $bc\frac{N}{\mu}$ in the diagonal entries in Equation \eqref{conjform} are multiples of $\frac{N}{\sigma_{H}}\big|K_{H}$ and $H$ is invariant under such translations (see Remark \ref{transK}), the latter assertion is equivalent to condition $(iii)$ of Lemma \ref{detnorm}. This proves the corollary.
\end{proof}

As mentioned in the Introduction, Theorem 2.1 of \cite{[IJK]} also gives a criterion for an element of $SL_{2}(\mathbb{R})$ (normalized differently, to land in $PGL_{2}^{+}(\mathbb{Q})$) to normalize $\Gamma_{H}$. In our language that theorem amounts to allowing one to restrict attention only to elements of $\Gamma_{0}^{*,\gcd\{s_{N},\eta_{H}\}}(N)$, a result which is clearly implied by Corollary \ref{inGammasigK} since $\frac{N}{K_{H}}\big|s_{N}$.

Corollary \ref{inGammasigK} presents the maximal $\sigma$ such that the normalizer of $\Gamma_{H}$ is contained in $\Gamma_{0}^{*,\sigma}(N)$. In the other direction, we will be interested to know when the smallest such group, namely the Atkin--Lehner group $\Gamma_{0}^{*}(N)$, is contained in that normalizer. For this we can prove the following result.
\begin{prop}
\begin{enumerate}[$(i)$]
\item The group $\Gamma_{0}^{*}(N)$ normalizes both $\Gamma_{0}(N)$ and $\Gamma_{1}(N)$. It operates on the quotient $(\mathbb{Z}/N\mathbb{Z})^{\times}$ via the quotient $\{\pm1\}^{\{p|N\}}$, in an explicit way.
\item $\Gamma_{0}^{*}(N)$ normalizes $\Gamma_{H}$ for a subgroup $H\subseteq(\mathbb{Z}/N\mathbb{Z})^{\times}$ if and only if $H$ is preserved under this action of $\{\pm1\}^{\{p|N\}}$.
\end{enumerate} \label{ALact}
\end{prop}

\begin{proof}
The fact that $\Gamma_{0}^{*}(N)$ normalizes both groups follows easily from Lemma \ref{detnorm}, or directly from the formula in Equation \eqref{conjform} (one can also observe that the kernel of the projection onto $\{\pm1\}^{\{p|N\}}$ is precisely $\Gamma_{0}(N)$). As $\Gamma_{0}(N)/\Gamma_{1}(N)$ is abelian, the action of $\Gamma_{0}^{*}(N)$ is via the quotient $\{\pm1\}^{\{p|N\}}$. Equation \eqref{conjform} shows that elements of $\Gamma_{0}^{*}(N)$ that are associated to the exact divisor $\mu$ of $N$ take an element $t\in(\mathbb{Z}/N\mathbb{Z})^{\times}$ to the element of $(\mathbb{Z}/N\mathbb{Z})^{\times}$ that is congruent to $t$ modulo $\frac{N}{\mu}$, but whose residue modulo $\mu$ is the one which is inverse to $t$. This proves part $(i)$, and part $(ii)$ immediately follows from it. This proves the proposition.
\end{proof}

Two types of natural subgroups of $(\mathbb{Z}/N\mathbb{Z})^{\times}$ that are of particular interest are the following ones. In case $H$ is the kernel of the projection onto $(\mathbb{Z}/D\mathbb{Z})^{\times}$ for a divisor $D$ of $N$, the group $\Gamma_{H}$ is the intersection $\Gamma_{0}(N)\cap\Gamma_{1}(D)$, which we denote $\Gamma_{0,1}(N,D)$. In particular $\Gamma_{0,1}(N,1)=\Gamma_{0}(N)$ and $\Gamma_{0,1}(N,N)=\Gamma_{1}(N)$. On the other hand, recall that the cardinality of $(\mathbb{Z}/N\mathbb{Z})^{\times}$ is given by \emph{Euler's totient function} $\varphi(N)=N\prod_{p|N}\big(1-\frac{1}{p}\big)$, while the exponent of that group is given by \emph{Carmichael's function} $\lambda(N)$. The value of the latter function is $\mathrm{lcm}\big\{\lambda(p^{v_{p}(N)})\big|p|N\big\}$, where on prime powers $p^{v_{p}(N)}$ Carmichael's function $\lambda$ coincides with $\varphi$, unless $p=2$ and $v_{2}(N)\geq3$ where it coincides with $\frac{\varphi}{2}$. Let $m$ be a divisor of $\lambda(N)$, and take $H=(\mathbb{Z}/N\mathbb{Z})^{\times}[m]$ to be the subgroup consisting of those elements of $(\mathbb{Z}/N\mathbb{Z})^{\times}$ whose order divides $m$. We denote the associated group $\Gamma_{H}$ by $\Gamma_{1}^{[m]}(N)$, so that $\Gamma_{1}^{[1]}(N)$ and $\Gamma_{1}^{[\lambda(N)]}(N)$ are just $\Gamma_{1}(N)$ and $\Gamma_{0}(N)$ respectively once more, and $\Gamma_{1}^{[2]}(N)$ is the group denoted $\Gamma_{1}^{\sqrt{1}}(N)$ in \cite{[LZ]}. Proposition \ref{ALact} therefore yields the following assertions:
\begin{cor}
$\Gamma_{0}^{*}(N)$ is contained in the normalizers of the following intermediate groups:
\begin{enumerate}[$(i)$]
\item $\Gamma_{1}^{[m]}(N)$ for any $m|\lambda(N)$.
\item $\Gamma_{0,1}(N,D)$ for any $D|N$.
\item $\Gamma_{H}$ for any subgroup $H$ of $(\mathbb{Z}/N\mathbb{Z})^{\times}[2]$.
\item Any group that is generated by groups of the sort considered in parts $(i)$, $(ii)$, and $(iii)$.
\end{enumerate} \label{ALinnorm}
\end{cor}

\begin{proof}
It suffices, by Proposition \ref{ALact}, to show that the corresponding groups are invariant under the action of $\{\pm1\}^{\{p|N\}}$ described explicitly in the proof of that proposition. Part $(i)$ thus follows from the fact that $(\mathbb{Z}/N\mathbb{Z})^{\times}[m]$ is a characteristic subgroup of $(\mathbb{Z}/N\mathbb{Z})^{\times}$ for any $m|\lambda(N)$. For part $(ii)$ we consider $\{\pm1\}^{\{p|D\}}$ as the quotient of $\{\pm1\}^{\{p|N\}}$ as above, and observe that the action of the former group on $(\mathbb{Z}/D\mathbb{Z})^{\times}$ and of the latter group on $(\mathbb{Z}/N\mathbb{Z})^{\times}$ commute with the projection map from residues modulo $N$ to residues modulo $D$. As this implies that the kernel of that projection is preserved under the action of $\{\pm1\}^{\{p|N\}}$, the assertion of part $(ii)$ follows. Part $(iii)$ is easily established since the operation of $\{\pm1\}^{\{p|N\}}$ is trivial on any element of $(\mathbb{Z}/N\mathbb{Z})^{\times}[2]$, and part $(iv)$ is an immediate consequence of the previous ones. This proves the corollary.
\end{proof}

\begin{rmk}
Proposition \ref{ambgrp} and Corollary \ref{ALinnorm} suffice to determine the normalizer of $\Gamma_{1}^{[m]}(N)$ for any divisor $m$ of $\lambda(N)$, as well as of $\Gamma_{0,1}(N,D)$ for divisor $D$ of $N$, as precisely $\Gamma_{0}^{*}(N)$ if $N$ is square-free (this will also follow from the more general results of Theorems \ref{Gamma01ND} and \ref{tormnorm} below). On the other hand, there are examples of groups $\Gamma_{H}$ whose normalizer does not contain $\Gamma_{0}^{*}(N)$, even in the square-free $N$ case: Consider the case where $N=91$ and $H$ is the group generated by an element $a$ whose images in both $\mathbb{F}_{7}$ and in $\mathbb{F}_{13}$ generate the multiplicative groups of the corresponding field. Since $a$ has order 12 in $(\mathbb{Z}/91\mathbb{Z})^{\times}$, any power of $a$ is determined by its image in $\mathbb{F}_{13}^{\times}$. But the image of $a$ under an element of $\Gamma_{0}^{*}(91)$ with $\mu=7$ is a residue that coincides with $a$ modulo 13 but not modulo 7. Hence this image is not a power of $a$ in $(\mathbb{Z}/91\mathbb{Z})^{\times}$, $H$ is not preserved under $\{\pm1\}^{\{p|N\}}$, and the normalizer of the associated congruence group of level 91 will be a proper subgroup of $\Gamma_{0}^{*}(91)$. In fact, as elements associated with $\mu=N$ (e.g., the \emph{Fricke involution}) operate on $(\mathbb{Z}/N\mathbb{Z})^{\times}$ via inversion, an argument similar to Proposition \ref{ambgrp} implies that such elements will normalize $\Gamma_{H}$ for any $H$. This determines the normalizer of our group with $N=91$, since the index of $\Gamma_{0}(91)$ in $\Gamma_{0}^{*}(91)$ is just 4 (see, e.g., part $(iii)$ of Proposition \ref{Gamma0*sigma}), so that Fricke subgroup has index 2 in $\Gamma_{0}^{*}(91)$, showing that it must be the normalizer of the congruence subgroup in question. \label{sqfandexc}
\end{rmk}

\section{Normalizers of Congruence Subgroups \label{CongSub}}

In this section we determine the normalizers of the several types of groups $\Gamma_{H}$, including the groups $\Gamma_{0,1}(N,D)$ with $D|N$ and $\Gamma_{1}^{[m]}(N)$ for $m|\lambda(N)$. We begin with the first family of congruence subgroups:
\begin{thm}
If $D$ is a divisor of $N$ then the normalizer of $\Gamma_{0,1}(N,D)$ is precisely $\Gamma_{0}^{*,\sigma}(N)$ for $\sigma=\gcd\big\{2D,\frac{N}{D}\big\}\cdot\frac{\gcd\{s_{N},24\}}{2^{\theta}\gcd\{s_{N},24,2D\}}$, where $\theta$ equals 1 if $2v_{2}(D)=v_{2}(N)-1$ and 0 otherwise. \label{Gamma01ND}
\end{thm}

\begin{proof}
The group $H$ is the kernel of the projection from $(\mathbb{Z}/N\mathbb{Z})^{\times}$ to $(\mathbb{Z}/D\mathbb{Z})^{\times}$. Before determining $\sigma_{H}$, we observe that the additional condition in Corollary \ref{inGammasigK} (i.e., condition $(iii)$ of Lemma \ref{detnorm}) is satisfied with this $H$ by any element of $\Gamma_{0}^{*,s_{N}}(N)$, since $ad\mu$ and $bc\frac{N}{\mu}$ are integers with difference 1: Indeed, taking residues modulo $D$ one may replace $e$ and $h$ by 1 and get the desired result (recall that the proof of Corollary \ref{inGammasigK} shows that we only need to verify that these elements lie in $H$, since in this case they will always be inverse to one another, so that indeed working only modulo $D$ suffices here). It follows that the normalizer of $\Gamma_{0,1}(N,D)$ is the full group $\Gamma_{0}^{*,\sigma_{H}}(N)$, and we need to show that $\sigma=\sigma_{H}$ has the asserted value. The number $K_{H}$ from Corollary \ref{inGammasigK} is $\mathrm{lcm}\{D,s_{N}t_{N}\}$, so that $\frac{N}{K_{H}}$ is $\gcd\big\{s_{N},\frac{N}{D}\big\}$.

We claim that $\eta_{H}=\mathrm{lcm}\big\{2D,\gcd\{N,24\}\big\}$, unless $\frac{N}{D}$ is odd, where $2D$ has to be replaced by $D$. The difference of any two inverses modulo $N$ which are congruent to 1 modulo $D$ is clearly divisible by $D$. We claim that if $\frac{N}{D}$ is even then it has to be divisible by $2D$. Indeed, if $D$ is odd and $N$ is even then all the diagonal entries of matrices in $\Gamma_{0,1}(N,D)$ are odd, hence congruent to 1 modulo $2D$. On the other hand, if both $D$ and $\frac{N}{D}$ are even and $1+kD$ and $1+lD$ are inverses modulo $N$ then $\frac{N}{D}$ divides $k+l+klD$, so that $k$ and $l$ must be of the same parity. Hence every such difference is divisible by $D$ if $\frac{N}{D}$ is odd but by $2D$ if $\frac{N}{D}$ is even. But such a difference is also divisible by $\gcd\{N,24\}$ since $(\mathbb{Z}/24\mathbb{Z})^{\times}$ has exponent 2. Indeed, if $\big(\begin{smallmatrix} e & f \\ g & h\end{smallmatrix}\big)$ is an element of $\Gamma_{0}(N)$ then $e$ and $h$ are residues which are inverse modulo $(N,24)$, and therefore $e$ and $h$ coincide modulo $(N,24)$ hence their difference is divisible by this number. As 24 is the maximal number $K$ such that $(\mathbb{Z}/K\mathbb{Z})^{\times}$ has exponent 2, a simple argument using the Chinese Remainder Theorem and examining residues modulo 9 or 16 shows that no number larger than the asserted value divides all the differences $e-h$ for $\big(\begin{smallmatrix} e & f \\ g & h\end{smallmatrix}\big)\in\Gamma_{0,1}(N,D)$. This determines $\eta_{H}$ as the asserted number, so that $\sigma_{H}$ is the $\gcd$ of $\frac{N}{D}$, $s_{N}$, and the value just determined of $\eta_{H}$. We may use $\mathrm{lcm}\big\{2D,\gcd\{N,24\}\big\}$ for $\eta_{H}$ in any case, since if $\frac{N}{D}$ is odd then either $s_{N}$ is odd or $v_{2}(D)>v_{2}(s_{N})$, so that this value yields the correct $\gcd$ with $s_{N}$ (hence the correct $\sigma_{H}$) in any case.

It remains to show that this $\gcd$ yields the asserted value. Replacing $\eta_{H}$ by simply $2D$ would give just the first multiplier divided by $2^{\theta}$, since $v_{p}(2D)$ and $v_{p}\big(\frac{N}{D}\big)$ cannot both exceed $v_{p}(s_{N})$ for any prime $p$, unless $p=2$ and the equality yielding $\theta=1$ holds. Now, if $v_{p}(2D) \geq v_{p}(s_{N})$ for some prime $p$ then the powers of $p$ dividing $\gcd\{s_{N},\eta_{H}\}$ and $\gcd\{s_{N},2D\}$ are both $p^{v_{p}(s_{N})}$, while if $v_{p}(2D) \geq v_{p}(24)$ (this is always the case for $p>3$) then $v_{p}(\eta_{H})=v_{p}(2D)$ and again $\gcd\{s_{N},\eta_{H}\}$ and $\gcd\{s_{N},2D\}$ are divisible by the same power of $p$. But those are, by definition, the primes $p$ which do not divide the quotient $\frac{\gcd\{s_{N},24\}}{\gcd\{s_{N},24,2D\}}$. On the other hand, if $v_{p}(2D)<\min\{v_{p}(s_{N}),v_{p}(24)\}$ then the power of $p$ which divides that quotient is the difference between the latter two expressions. But in this case we have $v_{p}\big(\frac{N}{D}\big)>v_{p}(s_{N})>v_{p}(2D)$ and $v_{p}(\eta_{H})=v_{p}(24)$, so that we compare $v_{p}(\sigma_{H})=\min\{v_{p}(s_{N}),v_{p}(24)\}$ with $v_{p}(2D)$ again. Therefore $\sigma_{H}=\gcd\big\{\frac{N}{D},s_{N},\eta_{H}\big\}$ equals the asserted value (since both are positive numbers which are divisible by the same prime powers), which completes the proof of the theorem.
\end{proof}

\begin{cor}
The index of the group $\Gamma_{0,1}(N,D)$ in its normalizer equals $D\sigma^{2}\prod_{p|D}\big(1-\frac{1}{p}\big)\prod_{p|\sigma,\ v_{p}(N)=2v_{p}(\sigma)}\big(1+\frac{1}{p}\big)\cdot\prod_{p|N/\sigma^{2}}2$, where $\sigma$ is the number from Theorem \ref{Gamma01ND}. \label{indinnorm}
\end{cor}

\begin{proof}
Part $(iii)$ of Proposition \ref{Gamma0*sigma} shows that the index of $\Gamma_{0}(N)$ in $\Gamma_{0}^{*,\sigma}(N)$ is $\sigma^{2}\prod_{p|\sigma,\ v_{p}(N)=2v_{p}(\sigma)}\big(1+\frac{1}{p}\big)\cdot\prod_{p|N/\sigma^{2}}2$. As the index of $\Gamma_{0,1}(N,D)$ in $\Gamma_{0}(N)$ is $\varphi(D)=D\prod_{p|D}\big(1-\frac{1}{p}\big)$ (since $H$ is the kernel of a surjective homomorphism onto $(\mathbb{Z}/D\mathbb{Z})^{\times}$), this proves the corollary.
\end{proof}

We therefore recover the results of \cite{[CN]}, \cite{[AL]}, \cite{[AS]}, \cite{[B]}, and others about the normalizers of the classical congruence groups:
\begin{cor}
The normalizer of $\Gamma_{1}(N)$ is $\Gamma_{0}^{*}(N)$. The normalizer of $\Gamma_{0}(N)$ is $\Gamma_{0}^{*,\gcd\{s_{N},24\}}(N)$. \label{Gamma1N0N}
\end{cor}

\begin{proof}
For $\Gamma_{1}(N)$ we take $D=N$ in Theorem \ref{Gamma01ND}. Then $\frac{N}{D}=1$ and $\gcd\{s_{N},24\}$ divides $2D$, so that $\sigma=1$ and the normalizer is $\Gamma_{0}^{*}(N)$. On the other hand, the result for $\Gamma_{0}(N)$ is obtained by taking $D=1$ in Theorem \ref{Gamma01ND}, which yields the value $\frac{\gcd\{2,N\}\gcd\{s_{N},24\}}{2^{\theta}\gcd\{s_{N},2\}}$ for $\sigma$. We have to show that this expression reduces to $\gcd\{s_{N},24\}$, i.e., the combination of the other three multipliers cancel to 1. But if $N$ is odd then all the multipliers are 1, if $v_{2}(N)=1$ then $2^{\theta}=2$ and $s_{N}$ is odd, and if $4|N$ then $2|s_{N}$ but $\theta=0$ once more. This proves the corollary.
\end{proof}

The discrepancy between our Corollary \ref{Gamma1N0N} and the main result of \cite{[L1]} in case $N=4$ arises from the fact that this reference considers subgroups of $PSL_{2}(\mathbb{Z})$. The group we must consider in this case is the one generated by $\Gamma_{1}(N)$ and $\{\pm I\}$. The result of \cite{[L1]} is thus recovered as a special case of the following proposition.
\begin{prop}
Let $N$ and $D$ be as in Theorem \ref{Gamma01ND}. Then the normalizer of $\pm\Gamma_{0,1}(N,D)$ coincides with that of $\Gamma_{0,1}(N,D)$, unless $N=D=4$ where it coincides with that of $\Gamma_{0}(4)$. \label{pmGamma01ND}
\end{prop}

\begin{proof}
Our group $H$ is the product of the image of $\pm1$ in $(\mathbb{Z}/N\mathbb{Z})^{\times}$ with the kernel of the projection to $(\mathbb{Z}/D\mathbb{Z})^{\times}$. Parts $(ii)$, $(iii)$, and $(iv)$ of Proposition \ref{ALinnorm} show that the normalizer must contain $\Gamma_{0}^{*}(N)$. On the other hand, the number $\eta_{H}$, which is based on divisibility of differences between inverse elements of $H$, is not affected by multiplication by $-1$. In addition, $K_{H}$ is based on the intersection of $H$ with the image of $1+s_{N}t_{N}\mathbb{Z}$ modulo $N$, and $-1$ is not there unless $s_{N}t_{N}$ is 1 or 2, i.e., unless $N$ is a divisor of 4. Moreover, the only case where $N|4$ and $-I\not\in\Gamma_{0,1}(N,D)$ is where $N=D=4$. This shows that the normalizer of $\pm\Gamma_{0,1}(N,D)$ is contained in the one of $\Gamma_{0,1}(N,D)$ unless $N=D=4$, and the reverse inclusion follows immediately from the centrality of $-I$ in $SL_{2}(\mathbb{R})$. As for $N=D=4$ we get $\pm\Gamma_{0,1}(N,D)=\Gamma_{0}(4)$, this completes the proof of the proposition.
\end{proof}

The assertion from \cite{[L1]} is just the following corollary.
\begin{cor}
The normalizer of $\pm\Gamma_{1}(N)$ is $\Gamma_{0}^{*}(N)$ if $N\neq4$, and it is $\Gamma_{0}^{*,2}(4)$ in case $N=4$. The latter group contains $\Gamma_{0}^{*}(N)$ as a subgroup of index 3. \label{pmGamma1N}
\end{cor}

\begin{proof}
This is just the case $D=N$ in Proposition \ref{pmGamma01ND}. The assertion thus follows from Corollary \ref{Gamma1N0N}, and the index is obtained by comparing the indices of $\Gamma_{0}(4)$ in the two groups, which are evaluated in Proposition \ref{Gamma0*sigma}. This proves the corollary.
\end{proof}

\smallskip

\begin{rmk}
In Theorem \ref{Gamma01ND} we consider congruences only on three of the entries of the matrices. But a simple argument allows us to extend the result to more general congruence subgroups: If $\Gamma$ is defined to be the group of matrices $\big(\begin{smallmatrix} a & b \\ c & d\end{smallmatrix}\big) \in SL_{2}(\mathbb{Z})$ in which $T|c$, $M|b$, and $a$ and $d$ are congruent to 1 modulo $D$ and $D|N=MT$, then $\Gamma$ is the subgroup of $SL_{2}(\mathbb{Z})$ which is obtained by conjugating $\Gamma_{0,1}(N,D)$ by $\frac{1}{\sqrt{M}}\big(\begin{smallmatrix} M & 0 \\ 0 & 1\end{smallmatrix}\big)$. Thus its normalizer in $SL_{2}(\mathbb{R})$ is obtained by conjugating the normalizer from Theorem \ref{Gamma01ND} by this matrix, and it contains $\Gamma$ with the index stated in Corollary \ref{indinnorm}. As a special case of this, Remark \ref{conjGamma*N/sigma2} implies that if $M=\sigma$ and $T=\frac{N}{\sigma}$ then the normalizer is just $\Gamma_{0}^{*}\big(\frac{N}{\sigma^{2}}\big)$. In addition, this determines the normalizer of the groups $\Gamma^{0}(N)$ and $\Gamma^{1}(N)$ (in which the condition on the $c$-entry is replaced by the same condition on the $b$-entry). Moreover, the special case in which $T=D=M$ reproduces the classical result, alluded to in \cite{[BKX]}, that the normalizer of $\Gamma(M)$ is precisely $SL_{2}(\mathbb{Z})$, since then $\sigma=M$ as well and we are in the situation related to Remark \ref{conjGamma*N/sigma2} with $\frac{N}{\sigma^{2}}=1$. \label{congall}
\end{rmk}

\smallskip

We now turn to subgroups defined by their exponents.
\begin{lem}
Let $N$ be an integer, let $p$ be a prime divisor of $N$, let $m$ be a divisor of $\lambda(N)$, and take $H=(\mathbb{Z}/N\mathbb{Z})^{\times}[m]$.
\begin{enumerate}[$(i)$]
\item The number $K_{H}$ from Corollary \ref{inGammasigK} equals $\frac{N}{\gcd\{m,s_{N}\}}$.
\item If $p$ is a prime divisor of $N$ and $\gcd\{p-1,m\}>2$ then $p$ does not divide the number $\eta_{H}$ from Corollary \ref{inGammasigK}.
\item In case $\gcd\{p-1,m\}$ is 1 or 2 and $p$ is odd, the power of $p$ which divides $\eta_{H}$ is $\max\{1,v_{p}(N)-v_{p}(m)\}$.
\item The power $v_{2}(\eta_{H})$ equals $\max\{3,v_{2}(N)-v_{2}(m)+1\}$ when $8|N$ and $m$ is even, and just $v_{2}(N)$ otherwise.
\end{enumerate} \label{tormdata}
\end{lem}

\begin{proof}
The kernel of the projection from $(\mathbb{Z}/N\mathbb{Z})^{\times}$ to $(\mathbb{Z}/s_{N}t_{N}\mathbb{Z})^{\times}$, namely $(1+s_{N}t_{N}\mathbb{Z})/N\mathbb{Z}$, is isomorphic to the cyclic group $s_{N}t_{N}\mathbb{Z}/N\mathbb{Z}$, of order $s_{N}$. Indeed, since $N|(s_{N}t_{N})^{2}$ we find that the product of $1+ks_{N}t_{N}$ with $1+ls_{N}t_{N}$ equals $1+(k+l)s_{N}t_{N}$ modulo $N$. In fact, $s_{N}t_{N}$ is the smallest divisor of $N$ with that property. Hence its $m$-torsion part is its torsion part of order $\gcd\{m,s_{N}\}$, which is generated by the image of $1+\frac{N}{\gcd\{m,s_{N}\}}$ modulo $N$. This proves part $(i)$.

We now recall from the Chinese Remainder Theorem that $(\mathbb{Z}/N\mathbb{Z})^{\times}$ decomposes as the product $\prod_{p|N}(\mathbb{Z}/p^{v_{p}(N)}\mathbb{Z})^{\times}$. Moreover, the components for odd primes $p$ are cyclic, while if $8|N$ then the component corresponding to $p=2$ is the product of $\{\pm1\}$ and a cyclic group (if $v_{2}(N)\leq2$ then this component is also cyclic, of order 1 or 2). This decomposition clearly goes over to a similar decomposition of $H=(\mathbb{Z}/N\mathbb{Z})^{\times}[m]$ as $\prod_{p|N}(\mathbb{Z}/p^{v_{p}(N)}\mathbb{Z})^{\times}\big[\gcd\big\{m,\lambda(p^{v_{p}(N)}\mathbb{Z})\big\}\big]$. Now, the condition from part $(ii)$ occurs only for odd $p\geq5$, and it implies that the image modulo $p$ of any generator of $(\mathbb{Z}/p^{v_{p}(N)}\mathbb{Z})^{\times}\big[\gcd\big\{m,\lambda(p^{v_{p}(N)}\mathbb{Z})\big\}\big]$, which must have order $\gcd\{p-1,m\}$ in $(\mathbb{Z}/p\mathbb{Z})^{\times}$, differs from its inverse modulo $p$. As this element $e$ and its inverse $h$ provide us a difference $e-h$ which is not divisible by $p$, part $(ii)$ follows.

On the other hand, if $\gcd\{p-1,m\}$ is 1 or 2 then the order of torsion inside $(\mathbb{Z}/p^{v_{p}(N)}\mathbb{Z})^{\times}$ in which we are interested is $p^{\min\{v_{p}(N)-1,v_{p}(m)\}}$ (perhaps multiplied by 2), except when $p=2$ and $v_{2}(N)\geq3$, where $\lambda(p^{v_{p}(N)})=\varphi(p^{v_{p}(N)})$ and the $-1$ has to be replaced by $-2$. Now, if $p$ is odd then elements of torsion order $p^{r}$ (resp. $2p^{r}$) for $r<v_{p}(N)$ inside $(\mathbb{Z}/p^{v_{p}(N)}\mathbb{Z})^{\times}$ are the images of $1+p^{v_{p}(N)-r}\mathbb{Z}$ (resp. $\pm1+p^{v_{p}(N)-r}\mathbb{Z}$) modulo $p^{v_{p}(N)}$, and the inverse of an element of the form $\pm1+ap^{v_{p}(N)-r}$ modulo $p^{v_{p}(N)}$ is congruent to $\pm1-ap^{v_{p}(N)-r}$ modulo $p^{v_{p}(N)-r+1}$. As this shows that their difference is divisible precisely by $p^{v_{p}(N)-r}$ if $a$ is not divisible by $p$, part $(iii)$ is also established by substituting $r=\min\{v_{p}(N)-1,v_{p}(m)\}$. For powers of 2 we note that if $8|N$ then the elements of order $2^{r}$ with $0<r<v_{p}(N)-1$ arise from $\pm1+2^{v_{2}(N)-r}\mathbb{Z}$, and the fact that the inverse of $\pm1+a\cdot2^{v_{2}(N)-r}$ with odd $a$ is $\pm1-a\cdot2^{v_{2}(N)-r}$ modulo $2^{v_{2}(N)-r+2}$ implies that the power of 2 which divides the corresponding difference is $v_{2}(N)-r+1$. This yields the assertion of part $(iv)$ (since $r$ is $\min\{v_{2}(N)-2,v_{2}(m)\}$ here), since the cases where $m$ is odd or where $v_{2}(N)\leq2$ are trivial. This proves the lemma.
\end{proof}

We can now prove our result about these types of groups.
\begin{thm}
For $N$ and $m$ as in Lemma \ref{tormdata}, we define $\sigma$ to be the number \[\prod_{\substack{p|\gcd\{m,s_{N}\} \\ \gcd\{p-1,m\}\leq2}}p^{\max\{1,\min\{v_{p}(m),v_{p}(2N)-v_{p}(m)\}\}}\cdot2^{\varepsilon-\max\{\theta,v_{2}(s_{N})\}},\] where $\theta$ is 1 if $2v_{2}(m)=v_{2}(N)+1$ and 0 otherwise (as in Theorem \ref{Gamma01ND}), and $\varepsilon$ equals 2 if $v_{2}(m) \geq v_{2}(N)\geq6$, 1 if $v_{2}(m)=v_{2}(N)-1\geq5$ or if $v_{2}(m)=v_{2}(N)\in\{4,5\}$, and 0 otherwise. Then the normalizer of $\Gamma_{1}^{[m]}(N)$ is precisely $\Gamma_{0}^{*,\sigma}(N)$. \label{tormnorm}
\end{thm}

\begin{proof}
First we prove that the number $\sigma_{H}$ from Corollary \ref{inGammasigK} is the asserted one. That corollary evaluates it as the $\gcd$ of the two numbers $\frac{N}{K_{H}}$ and $\eta_{H}$, both of which are evaluated explicitly in Lemma \ref{tormdata}. Clearly only primes $p$ which divide $\gcd\{m,s_{N}\}$, which equals $\frac{N}{K_{H}}$ by part $(i)$ of Lemma \ref{tormdata}, can divide $\sigma_{H}$, and the condition $\gcd\{p-1,m\}\leq2$ is also imposed since part $(ii)$ of Lemma \ref{tormdata} shows that primes not satisfying it do not divide $\eta_{H}$. Now, if $p$ is an odd prime such that $\gcd\{p-1,m\}\leq2$ then we deduce from parts $(i)$ and $(iii)$ of Lemma \ref{tormdata} that $v_{p}(\sigma_{H})$ is the minimum of $v_{p}(m)$, $v_{p}(s_{N})$, and $\max\{1,v_{p}(N)-v_{p}(m)\}$, all of which are at least 1. We may omit $v_{p}(s_{N})$ since either $v_{p}(m)$ or $v_{p}(N)-v_{p}(m)$ do not exceed it (and the inequality $v_{p}(s_{N})\geq1$ is obvious), and it is easy to see that for $v_{p}(m)>0$ the minimum of the remaining two numbers is $\max\{1,\min\{v_{p}(m),v_{p}(2N)-v_{p}(m)\}\}$, where the extra multiplier of 2 does not affect the value of $v_{p}$ since $p$ is odd.

We still need to determine the power of 2 which divides $\sigma_{H}$, in case both $m$ and $s_{N}$ are even (the $\gcd$ condition is always satisfied). Parts $(i)$ and $(iv)$ of Lemma \ref{tormdata} determine it as the minimum of Corollary \ref{inGammasigK} $v_{2}(m)$, $v_{2}(s_{N})$, and $\max\{3,v_{2}(2N)-v_{2}(m)\}$, unless $v_{2}(N)=2$ (it cannot be smaller if $2|s_{N}$) and the third number is 2. We have to show that this minimum coincides with $\max\{1,\min\{v_{2}(m),v_{2}(2N)-v_{2}(m)\}\}-\theta$ up to the asserted discrepancy $\varepsilon$ (the exponent of 2 includes $\max\{v_{2}(s_{N}),\theta\}$ rather than just $\theta$ in order to exclude the case with $v_{2}(N)=v_{2}(m)=1$, where $\theta=1$ and 2 does not divide $s_{N}$). In the case where $v_{2}(m) \leq v_{2}(N)-2$ (so that taking the maximum with 3 does not change $v_{2}(2N)-v_{2}(m)$), the proof of Theorem \ref{Gamma01ND} shows that  there is no discrepancy. The case where $v_{2}(s_{N})=1$ also yields discrepancy 0, since both final numbers equal 1. Now, if $v_{2}(m)=v_{2}(N)-1\geq3$ then we have to compare $\min\{3,v_{2}(s_{N})\}$ with 2, which corresponds with $\varepsilon$ being 1 if $v_{2}(N)\geq6$ but 0 when $v_{2}(s_{N})=2$. Finally, if $v_{2}(m) \geq v_{2}(N)$ then we have the discrepancy between $\min\{3,v_{2}(s_{N})\}$ and 1, which equals 2 in case $v_{2}(N)\geq6$ and 1 if $v_{2}(s_{N})=2$. This completes the determination of $\sigma_{H}$ as the asserted value.

It remains to consider the condition from Corollary \ref{inGammasigK}. But the proof of Lemma \ref{tormdata} shows that when we consider residues modulo $p^{v_{p}(N)}$, one of two situations may occur: Either $p$ does not divide $\sigma_{H}$, or the residues of the entries $e$ and $h$ lie in $H$ if and only if they are congruent to 1 (or to $\pm1$) modulo $p^{r}$ for some fixed power $r$. In the first case the numbers $ade\mu-bch\frac{N}{\mu}$ and $adh\mu-bce\frac{N}{\mu}$ are either $e$ and $h$ or $h$ and $e$ modulo $p^{v_{p}(N)}$ (depending on whether $p$ divides $\frac{N}{\mu}$ or $\mu$), while in the second case we take the residue modulo $p^{r}$ and obtain that our numbers are congruent to $e$ and to $h$ modulo $p^{r}$ as well. As both operations preserve the $m$-torsion modulo $p^{v_{p}(N)}$ (as in the proof of Theorem \ref{Gamma01ND} and Proposition \ref{pmGamma01ND}), hence also modulo $N$ by the Chinese Remainder Theorem, we deduce from Corollary \ref{inGammasigK} that the full group $\Gamma_{0}^{*,\sigma}(N)$ normalizes $\Gamma_{1}^{[m]}(N)$. This proves the theorem.
\end{proof}

\begin{rmk}
The appearance of the maximum (with 1 or 3) in Lemma \ref{tormdata} and Theorem \ref{tormnorm} is not redundant. There exist cases where some prime $p$ may divide $2N$ (and even $s_{N}$), but can divide $m$ to a larger power. To give examples, consider $N=68$, with $v_{2}(N)=2$ but whose $\lambda$-value 16 has $v_{2}=4$, or $N=9\cdot163$, where $v_{3}(N)=2$ but the divisor $163-1$ of $\lambda(N)$ is divisible by $3^{4}$. \label{maxmust}
\end{rmk}

The special case of prime $m$ in Theorem \ref{tormnorm} is of particular interest.
\begin{cor}
If $N$ is a number such that $\lambda(N)$ is divisible by a prime number $l$ then the normalizer of $\Gamma_{1}^{[l]}(N)$ is $\Gamma_{0}^{*,l}(N)$ if $l|s_{N}$ and just $\Gamma_{0}^{*}(N)$ otherwise. \label{torprime}
\end{cor}

\begin{proof}
The only prime we must consider in Theorem \ref{tormnorm} is $p=l$, and only in the case $l|s_{N}$. The $\gcd$ condition is immediate, and it is clear that the power is $v_{l}(l)=1$ (also when $l=2$, since $\theta=0$ if $2|s_{N}$ and we are in the situation where $v_{2}(m)<5$, hence $\varepsilon=0$). The assertion thus follows from Theorem \ref{tormnorm}. This proves the corollary.
\end{proof}

Our results can now be used to determine the normalizer of $\Gamma_{H}$ for any subgroup $H$ of $(\mathbb{Z}/N\mathbb{Z})^{\times}$ in case $N$ is a prime power.
\begin{prop}
Let $l$ be a prime, and take $N=l^{u}$ for some integer $u$.
\begin{enumerate}[$(i)$]
\item If $l$ is odd then any group between $\Gamma_{1}(N)$ and $\Gamma_{0}(N)$ is of the form $\Gamma_{1}^{[m]}(N)$ for some divisor $m$ of $\lambda(N)=\varphi(N)=(l-1)l^{u-1}$, which is of the form $kl^{w}$ for some $k|l-1$ and $0 \leq w<u$. The corresponding normalizer is $\Gamma_{0}^{*}(N)$ if $k>2$ and is $\Gamma_{0}^{*,\sigma}(N)$ for $\sigma=l^{\min\{w,u-w\}}$ if $k\leq2$.
\item If $l=2$ then any group between $\pm\Gamma_{1}(N)$ and $\Gamma_{0}(N)$ can be considered as $\Gamma_{1}^{[m]}(N)$ for $m=2^{w}$ with $w \leq u-2$, where if $8|N$ then $\pm\Gamma_{1}(N)$ itself is associated with $m=1$ (even though it does not equal $\Gamma_{1}^{[1]}(N)=\Gamma_{1}(N)$). When $u\leq2$ we take $m=u-1$. The normalizer then equals $\Gamma_{0}^{*,\sigma}(N)$, with $\sigma$ being $2^{\min\{w,u+1-w\}-\theta}$ where $\theta$ is 1 if $2w=u+1$ and 0 otherwise.
\item In case $l=2$, $u\geq2$, and the group $H$ contains only elements which are congruent to 1 modulo 4 then our group is of the form $\Gamma_{0,1}(N,D)$ for some $D=2^{u-w}$ with $w \leq u-2$. In this case the normalizer is $\Gamma_{0}^{*,\sigma}(N)$ for $\sigma=2^{\min\{w,u+1-w\}-\theta}$ with $\theta$ as in part $(ii)$.
\item The remaining case is where $u\geq3$ and the group is generated by some element which is congruent to $-1$ modulo 4. Such a group
    intersects the image of $1+s_{N}t_{N}\mathbb{Z}$ in the kernel of the projection to $(\mathbb{Z}/2^{u-w}\mathbb{Z})^{\times}$ for some $0 \leq w \leq u-3$, and the normalizer is $\Gamma_{0}^{*,\sigma}(N)$ where $\sigma$ is just $2^{\min\{w,u-w\}}$.
\end{enumerate} \label{primpow}
\end{prop}

\begin{proof}
The first assertion in part $(i)$ follows from the cyclicity of $(\mathbb{Z}/N\mathbb{Z})^{\times}$ for $N$ an odd prime power. Examining the group described in Theorem \ref{tormnorm}, we find that the only prime which may divide $s_{N}$ is $p=l$, and the corresponding $\gcd$ is $k$. The assertion is thus established in case $k>2$, and if $k\leq2$ we see that the difference $v_{p}(N)-v_{p}(m)=u-w$ is at least 1, so that the minimum is $\geq1$ if $l|m$ and 0 otherwise (the case $k=1$ can also be established by taking $D=p^{u-w}>1$ in Theorem \ref{Gamma01ND}, since the quotient appearing in that theorem is trivial wherever $D$ and $N$ are non-trivial powers of the same odd prime). This establishes part $(i)$ since $v_{l}(2N)=u$ for odd $l$. For the remaining parts we recall the structure of $(\mathbb{Z}/N\mathbb{Z})^{\times}$ for $N=2^{u}$ with $u\geq2$ (if $u=1$ then this group is trivial) as $\{\pm1\}$ times the cyclic group of the residues which are congruent to 1 modulo 4. This proves the first assertions in all the remaining three parts, so that part $(iii)$ follows from Theorem \ref{Gamma01ND} since the quotient appearing there is trivial for $s_{N}$ a power of 2 and $D$ divisible by 4. Part $(ii)$ is now a consequence of Theorem \ref{tormnorm} since $\varepsilon=0$ for $1 \leq v_{2}(m) \leq v_{2}(N)-2$ there, complemented by the case where $H$ is just the image of $\{\pm1\}$ to which we associated the value $m=1$ for $N\neq4$ and $m=2$ for $N=4$ (this case is proved in Corollary \ref{pmGamma1N}). In part $(iv)$ we recall that our group contains, apart from the kernel of the projection to $(\mathbb{Z}/2^{u-w}\mathbb{Z})^{\times}$, also elements which are to $-1$ precisely modulo $2^{u-1-w}$. The same argument from the proof of Theorem \ref{Gamma01ND} shows that the normalizer is $\Gamma_{0}^{*,\sigma}(N)$ where $\sigma$ is 2 raised to the power which is the minimum of $w$, $v_{2}(s_{N})$, and $\max\{3,u-w\}$ (since we use $D=2^{u-w}$ for finding $K_{H}$ but $D=2^{u-1-w}$ for determining $\eta_{H}$). One then deduces that $\sigma=2^{\min\{w,u-w\}}$, since $u-w\geq3$ and $v_{2}(s_{N})$ is not smaller than either $w$ or $u-w$. This proves the proposition.
\end{proof}

\smallskip

The groups considered in Theorems \ref{Gamma01ND} and \ref{tormnorm} can be combined to yield, e.g., groups of elements of $(\mathbb{Z}/N\mathbb{Z})^{\times}$ whose $m$th power is trivial modulo $D$. We shall not carry out the examination of these groups in general, but just mention that in the case $m=2$, which is related to lattices by Theorem \ref{autlat} below, we have $D|\eta_{H}$ and $K_{H}=\mathrm{lcm}\big\{s_{N}t_{N},\frac{D}{\gcd\{2,D\}}\big\}$ (the denominator making sure that only integers are considered). Hence $\sigma_{H}$ is divisible by $\sigma=\gcd\big\{D,\frac{2N}{D}\big\}/2^{\theta}$, with $\theta$ being 1 if $2v_{2}(D)=v_{2}(N)+1$ and 0 otherwise (and it probably equals that number, except perhaps for a few small values of $v_{2}(D)$), and one easily verifies that $\Gamma_{0}^{*,\sigma}(N)$ is contained in the normalizer in question. Indeed, Theorem \ref{autlat} below shows that the corresponding group $\Gamma_{H}$ will be the discriminant kernel of a lattice $L(N,D)$, whose group of automorphisms (arising from $SL_{2}(\mathbb{R})$) is $\Gamma_{0}^{*,\sigma}(N)$. The general case, as well as the fine details of this case, are left for future investigation.

\section{Lattices \label{Lat}}

Consider the space $M_{2}(\mathbb{R})_{0}$ of traceless $2\times2$ real matrices. With the bilinear form $(X,Y)=Tr(XY)$ (so that $(X,X)=-2\det X$) it becomes a real quadratic space of signature $(2,1)$. The space $M_{2}(\mathbb{R})_{0}$ has a convenient basis, consisting of the three matrices \[E=\begin{pmatrix} 0 & 1 \\ 0 & 0\end{pmatrix},\quad H=\begin{pmatrix} 1 & 0 \\ 0 & -1\end{pmatrix},\quad\mathrm{and}\quad F=\begin{pmatrix} 0 & 0 \\ 1 & 0\end{pmatrix}.\] $E$ and $F$ span a hyperbolic plane, and $H$ is orthogonal to both of them and pairs to 2 with itself.

The connected component $SO^{+}\big(M_{2}(\mathbb{R})_{0}\big)$ of the orthogonal group of this vector space is isomorphic to $PSL_{2}(\mathbb{R})$, as one sees by letting $SL_{2}(\mathbb{R})$ act on $M_{2}(\mathbb{R})_{0}$ by conjugation. The action of an element $A=\big(\begin{smallmatrix} e & f \\ g & h\end{smallmatrix}\big) \in SL_{2}(\mathbb{R})$, whose inverse $A^{-1}=\big(\begin{smallmatrix} h & -f \\ -g & e\end{smallmatrix}\big)$, sends our basis elements to
\begin{equation}
\begin{pmatrix} -eg & e^{2} \\ -g^{2} & eg\end{pmatrix},\quad\begin{pmatrix} eh+fg & -2ef \\ 2gh & -eh-fg\end{pmatrix},\quad\mathrm{and}\quad\begin{pmatrix} fh & -f^{2} \\ h^{2} & -fh\end{pmatrix} \label{action}
\end{equation}
respectively. Hence using the basis $-E$, $\frac{1}{2}H$, and $F$ we get the natural formula for the action of the symmetric square representation.

\smallskip

Let $N\in\mathbb{N}$ and a divisor $D$ of $N$ be given. We define $L(N,D)$ to be the lattice spanned by $\frac{\sqrt{D}}{\sqrt{N}}E$, $\sqrt{ND} \cdot F$, and $\frac{\sqrt{N}}{\sqrt{D}}H$. It is isomorphic to the orthogonal direct sum of a hyperbolic plane rescaled by $D$ and a 1-dimensional lattice generated by a vector of norm $\frac{2N}{D}$. The lattice $L_{0}(N)=L(N,1)$ is the one considered in \cite{[BO]}, while \cite{[LZ]} considers the lattice $L(N,N)$, which is denoted there $L_{1}(N)$. The dual lattice $L^{*}(N,D)=Hom\big(L(N,D),\mathbb{Z}\big)$ is identified as a subgroup of $M_{2}(\mathbb{R})_{0}$ via the bilinear form, and it is generated by $\frac{1}{\sqrt{DN}}E$, $\frac{\sqrt{N}}{\sqrt{D}}F$, and $\frac{\sqrt{D}}{2\sqrt{N}}H$. We wish to determine pre-image of the group $SAut^{+}\big(L(N,D)\big)$ (the group of automorphisms of $L(N,D)$ which lie in the connected component $SO^{+}\big(M_{2}(\mathbb{R})_{0}\big)$) in $SL_{2}(\mathbb{R})$, as well as its \emph{discriminant kernel}, i.e., the subgroup of $SAut^{+}\big(L(N,D)\big)$ which operates trivially on the discriminant group $\Delta(N,D)=L^{*}(N,D)/L(N,D)$ (or more precisely the pre-image in $SL_{2}(\mathbb{R})$ of that group as well). As for the lattices, the discriminants $\Delta(N,N)$ and $\Delta(N,1)$ will also be denoted $\Delta_{0}(N)$ and $\Delta_{1}(N)$ respectively.

The determination of the pre-image of $SAut^{+}\big(L(N,D)\big)$ in $SL_{2}(\mathbb{R})$ begins with the following observation.
\begin{lem}
Any matrix in $SL_{2}(\mathbb{R})$ whose action preserves $L(N,D)$ must lie in $\Gamma_{0}^{*,s_{N}}(N)$. \label{grplat}
\end{lem}

\begin{proof}
We know that rescaling of a lattice, i.e., multiplication of all of its generators by the square root of some integer, leaves the group $SAut^{+}\big(L(N,D)\big)$ invariant. We therefore rescale $L(N,D)$ by $\frac{N}{D}$, and obtain the lattice generated by $E$, $NF$, and $\frac{N}{D}H$. All of these lattices are contained in $L_{1}(N)$ (with generators $E$, $NF$, and $H$), and contain the rescaling $\widetilde{L}_{0}(N)$ of $L_{0}(N)$ by $N$ (generators of which can be taken to be $E$, $NF$, and $NH$). It follows that if the action of the matrix $A=\big(\begin{smallmatrix} e & f \\ g & h\end{smallmatrix}\big) \in SL_{2}(\mathbb{R})$ preserves $L(N,D)$ then it must take $\widetilde{L}_{0}(N)$ into $L_{1}(N)$. But the operation of $A$ sends the basis of $\widetilde{L}_{0}(N)$ to the matrices given in Equation \eqref{action}, with the middle and right matrices multiplied by $N$. The resulting matrices lie in $L_{1}(N)$ if and only if the expressions $e^{2}$, $eg$, $\frac{g^{2}}{N}$, $2Nef$, $N(eh+fg)$, $2gh$, $Nf^{2}$, $Nfh$, and $h^{2}$, which represent these matrices in terms of the basis for $L_{1}(N)$, are all integral. As these numbers are precisely the ones appearing in Lemma \ref{intent}, we deduce from Proposition \ref{ambgrp} that $A$ must be in $\Gamma_{0}^{*,s_{N}}(N)$. This proves the lemma.
\end{proof}

\smallskip

We can now prove our main result concerning automorphisms of lattices. We shall allow ourselves the abuse of notation denoting by $SAut^{+}\big(L(N,D)\big)$ also the subgroup of $SL_{2}(\mathbb{R})$ which lies over the automorphism group (which is a subgroup of $PSL_{2}(\mathbb{R})$ by definition). The same applies for its subgroups, in particular the discriminant kernel. In addition, note that $\Delta(N,D)$ decomposes as the direct sum of three cyclic groups, namely $\Delta(N,D)_{H}$ that is generated by $\frac{\sqrt{D}}{2\sqrt{N}}H+L(N,D)$ (of order $\frac{2N}{D}$), $\Delta(N,D)_{E}$ with generator $\frac{1}{\sqrt{DN}}E+L(N,D)$ (having order $D$), and $\Delta(N,D)_{F}$, a generator for which is $\frac{\sqrt{N}}{\sqrt{D}}F+L(N,D)$ (with order $D$ as well). The notation $\Delta_{0}(N)_{H}$, $\Delta_{1}(N)_{H}$ etc. is defined similarly.
\begin{thm}
\begin{enumerate}[$(i)$]
\item $SAut^{+}\big(L(N,D)\big)$ is the group $\Gamma_{0}^{*,\sigma}(N)$, where $\sigma$ is defined to be $\gcd\big\{D,\frac{2N}{D}\big\}/2^{\theta}$, with $\theta$ being 1 if $2v_{2}(D)=v_{2}(N)+1$ and 0 otherwise.
\item The stabilizer in $SAut^{+}\big(L(N,D)\big)$ of the subgroup $\Delta(N,D)_{H}$ of $\Delta(N,D)$, or equivalently of its orthogonal complement $\Delta(N,D)_{E}\oplus\Delta(N,D)_{F}$ is precisely $\Gamma_{0}^{*}(N)$.
\item The subgroup of $SAut^{+}\big(L(N,D)\big)$ that fixes all the elements of $\Delta(N,D)_{H}$ pointwise in the discriminant group is the subgroup of $\Gamma_{0}^{*}(N)$, containing $\Gamma_{0}(N)$, that is based only on those divisors $\mu$ of $N$ which are co-prime to $\frac{N}{D}$.
\item Elements of $SAut^{+}\big(L(N,D)\big)$ whose action leaves all the three subgroups $\Delta(N,D)_{H}$, $\Delta(N,D)_{E}$, and $\Delta(N,D)_{F}$ invariant is the group consisting of those elements of $\Gamma_{0}^{*}(N)$ in which $\mu$ is co-prime to $D$.
\item The discriminant kernel of $L(N,D)$ is the subgroup $\Gamma_{H}$ of $\Gamma_{0}(N)$, in which $H$ is group consisting of those elements of $(\mathbb{Z}/N\mathbb{Z})^{\times}$ whose square becomes trivial in $(\mathbb{Z}/D\mathbb{Z})^{\times}$.
\end{enumerate} \label{autlat}
\end{thm}

\begin{proof}
As in the proof of Lemma \ref{grplat}, we determine the $SAut^{+}$ group of the rescaled lattice. Similar considerations show that using the basis $E$, $NF$, and $\frac{N}{D}H$ of that lattice (rather than the generators $E$, $NF$, and $NH$ of $\widetilde{L}_{0}(N)$), and requiring integrality with respect to the same basis (instead of with respect to $E$, $NF$, and $H$ spanning $L_{1}(N)$), we have to consider elements $A=\big(\begin{smallmatrix} e & f \\ g & h\end{smallmatrix}\big)$ of $\Gamma_{0}^{*,s_{N}}(N)$ in which $e^{2}$, $\frac{Deg}{N}$, $\frac{g^{2}}{N}$, $\frac{2Nef}{D}$, $eh+fg$, $\frac{2gh}{D}$, $Nf^{2}$, $Dfh$, and $h^{2}$ are integers. Writing $A$ as in Definition \ref{Gamma0*sNdef}, we find that the numbers which must be integral are $a^{2}\mu$, $Dac$, $c^{2}\frac{N}{\mu}$, $\frac{2N}{D}ab$, $ad\mu+bc\frac{N}{\mu}$, $\frac{2N}{D}cd$, $b^{2}\frac{N}{\mu}$, $Dbd$, and $a^{2}\mu$. This is not immediate only for the expressions not involving $\mu$, which are invariants of the matrix $A$ (i.e., independent of the presentation) by part $(i)$ of Lemma \ref{nocanN}, and which determine the group $\Gamma_{0}^{*,\sigma}(N)$ to which $A$ belongs by Remark \ref{Gamma*sigmaeq}. This determines the $SL_{2}(\mathbb{R})$-pre-image of $SAut^{+}\big(L(N,D)\big)$ as $\Gamma_{0}^{*,\sigma}(N)$, where $\sigma$ is $\gcd\big\{D,\frac{2N}{D},s_{N}\}$. The same considerations as in the proof of Theorem \ref{Gamma01ND} show that $\sigma$ has the asserted value (but note the difference that here $\frac{N}{D}$ is multiplied by 2, while the number which is multiplied by 2 in that theorem is $D$, whence the difference in the definition of $\theta$). This proves part $(i)$.

For the remaining parts it will be convenient to rescale $L^{*}(N,D)$ as well, to get the lattice generated by $\frac{1}{D}E$, $\frac{1}{2}H$, and $\frac{N}{D}F$ (this is \emph{not} the dual of the rescaled lattice, but just makes the coefficients in the calculation neater). An element $A\in\Gamma_{0}^{*,\sigma}(N)$, written as in  Definition \ref{Gamma0*sNdef} (or \ref{Gamma0*sigmadef}), takes this basis to \[\frac{1}{D}\!\begin{pmatrix} -acN & a^{2}\mu \\ -c^{2}\frac{N^{2}}{\mu^{2}} & acN\end{pmatrix},\quad\!\frac{1}{2}\!\begin{pmatrix} ad\mu+bc\frac{N}{\mu} & -2ab \\ 2cd & -ad\mu-bc\frac{N}{\mu}\end{pmatrix},\quad\!\mathrm{and}\quad\!\frac{N}{D}\!\begin{pmatrix} bd & -\frac{b^{2}}{\mu} \\ d^{2}\mu & -bd\end{pmatrix}\] respectively. The action of $A$ preserves $\Delta(N,D)_{H}$ (and $\Delta(N,D)_{E}\oplus\Delta(N,D)_{F}$) if and only if the parts of the image of $\frac{1}{2}H$ that is spanned by $\frac{1}{D}E$ and $\frac{N}{D}F$, as well as those of the images of $\frac{1}{D}E$ and $\frac{N}{D}F$ that are multiples of $\frac{1}{2}H$, belong to the rescaled $L(N,D)$. This means that they are spanned by $E$, $NF$, and $\frac{N}{D}H$, and this happens if and only if $ac$, $ab$, $cd$, and $bd$ are integral. As this is the description of elements of $\Gamma_{0}^{*}(N)$ (see, e.g., Remark \ref{Gamma*sigmaeq} again), this establishes part $(ii)$.

We may therefore assume, if our $A$ is given in the form from Lemma \ref{nocanN}, that $a$, $b$, $c$, and $d$ are integers. Then the generator of $\Delta(N,D)_{H}$ is multiplied by the image of $ad\mu+bc\frac{N}{\mu}$ modulo $\frac{2N}{D}$, and this number is congruent to 1 modulo $\frac{2N}{\mu}$ and to $-1$ modulo $2\mu$ by the $SL_{2}$ condition. As this operation on residues modulo $\frac{2N}{D}$ is via a faithful action of the quotient from part $(ii)$ of Lemma \ref{pm1pdivN} (with the divisor being $\frac{N}{D}$), the pointwise stabilizer of $\Delta(N,D)_{H}$ corresponds to the kernel of the projection map from that lemma, yielding part $(iii)$. On the other hand, the condition for elements of $\Gamma_{0}^{*}(N)$ to leave the two subgroups $\Delta(N,D)_{E}$ and $\Delta(N,D)_{F}$ invariant as well is the divisibility of $b^{2}\frac{N}{\mu}$ and $c^{2}\frac{N}{\mu}$ by $D$. But a prime dividing $\mu$ cannot divide $b$, $c$, or $\frac{N}{\mu}$, so that such divisibility condition can (and will) hold if and only if no such prime divides $D$, and part $(iv)$ follows.

Now, as elements of the discriminant kernel must satisfy the conditions of both parts $(iii)$ and $(iv)$, and only $\mu=1$ is co-prime to both $D$ and $\frac{N}{D}$, we deduce from these parts that the discriminant kernel is contained in $\Gamma_{0}(N)$. But an element $A=\big(\begin{smallmatrix} a & b \\ Nc & d\end{smallmatrix}\big)$ of that group fixes the generator of $\Delta(N,D)_{H}$, but multiplies the generators of $\Delta(N,D)_{E}$ and $\Delta(N,D)_{F}$ by $a^{2}$ and $d^{2}$ respectively (this generalizes the assertion for $D=N$ given in Proposition 4.1 of \cite{[LZ]}). As the latter subgroups are cyclic of order $D$, the discriminant kernel consists of those elements of $\Gamma_{0}(N)$ as above that satisfy $a^{2} \equiv d^{2}\equiv1(\mathrm{mod\ }D)$. As this is precisely the asserted group $\Gamma_{H}$, we obtain part $(v)$ as well. This completes the proof of the theorem.
\end{proof}

The results of Proposition 2.2 of \cite{[BO]} and of Remark 4.3 of \cite{[LZ]} are obtained as special cases of Theorem \ref{autlat}.
\begin{cor}
\begin{enumerate}[$(i)$]
\item The group $SAut^{+}\big(L_{0}(N)\big)$, as well as its separating subgroups from parts $(ii)$ and $(iv)$ of Theorem \ref{autlat}, is $\Gamma_{0}^{*}(N)$. The pointwise stabilizer of $\Delta_{0}(N)_{H}$, which is also the discriminant kernel, is just $\Gamma_{0}(N)$.
\item For $L_{1}(N)$, the $SAut^{+}$ group is $\Gamma_{0}^{*,2}(N)$ if $4|N$ and $\Gamma_{0}^{*}(N)$ otherwise. The stabilizer, as well as the pointwise stabilizer, of $\Delta_{1}(N)_{H}$ is $\Gamma_{0}^{*}(N)$. The separating subgroup from part $(iv)$ of Theorem \ref{autlat} is $\Gamma_{0}(N)$, and the discriminant kernel is $\Gamma_{1}^{[2]}(N)$ (or $\Gamma_{1}^{\sqrt{1}}(N)$ in the notation of \cite{[LZ]}).
\end{enumerate} \label{L0NL1N}
\end{cor}

\begin{proof}
Part $(i)$ is the case $D=1$ in Theorem \ref{autlat}, where $\Delta_{0}(N)_{E}$ and $\Delta_{0}(N)_{F}$ are trivial, so that $\Delta_{0}(N)_{H}$ is the full discriminant group $\Delta_{0}(N)$. For part $(ii)$ we take $D=N$ in that theorem, so that $\sigma$ from part $(i)$ there is $\frac{\gcd\{N,2\}}{2^{\theta}}$. As the numerator is 1 for odd $N$ and $\theta=1$ in case $v_{2}(N)=1$, this proves the first assertion. The rest follows directly from Theorem \ref{autlat}, noting that $\Delta_{1}(N)_{H}$ is of order 2 and hence has no non-trivial automorphisms. This proves the corollary.
\end{proof}

\smallskip

\begin{rmk}
Similarly to the groups considered in Remark \ref{congall}, we may generalize the lattices $L(N,D)$ to any lattice that is spanned by three vectors of the form $\frac{\sqrt{DM}}{\sqrt{T}}E$, $\frac{\sqrt{DT}}{\sqrt{M}}F$, and $\frac{\sqrt{MT}}{\sqrt{D}}H$, with $M$, $T$, and $D$ integers such that $D|N=MT$. As such a lattice is the image of $L(N,D)$ under the element $\frac{1}{\sqrt{M}}\big(\begin{smallmatrix} M & 0 \\ 0 & 1\end{smallmatrix}\big)$ of $SL_{2}(\mathbb{R})$, all of the associated groups from Theorem \ref{autlat} are obtained from those of $L(N,D)$ via conjugation by this matrix. In particular the rescaling of the lattice $L(N,D)$ by some integer $K$ (or even a rational number $K$, provided that the rescaling remains an even lattice) yields such a conjugate of $L(NK^{2},DK)$. Once again the special case in which $T=D=M$ is related to $\Gamma(M)$: This lattice is the rescaling of $L_{0}(1)=L_{1}(1)=L(1,1)$ by $M$, its $SAut^{+}$ group is $SL_{2}(\mathbb{Z})$, and its discriminant kernel is not precisely $\Gamma(M)$ but the congruence subgroup consisting of those matrices in $SL_{2}(\mathbb{Z})$ whose image in $SL_{2}(\mathbb{Z}/M\mathbb{Z})$ is diagonal. \label{conjlat}
\end{rmk}

\smallskip

While Remark \ref{conjlat} extends the class of lattices $L(N,D)$ by including some of their isomorphic copies in $M_{2}(\mathbb{R})_{0}$, the lattices $L(N,D)$ themselves are mutually non-isomorphic.
\begin{prop}
Let $N$, $M$, $D$ and $C$ be integers such that $D|N$ and $C|M$. If $L(N,D)$ are $L(M,C)$ are isomorphic as lattices, then $M=N$ and $C=D$. \label{LNDnotiso}
\end{prop}

\begin{proof}
We need to show that the isomorphism class of $L(N,D)$ determines $N$ and $D$. First, $\gcd\big\{D,\frac{N}{D}\big\}$ is the minimal integer $K$ such that $L(N,D)$ is isomorphic to the rescaling of an even lattice (namely $L\big(\frac{N}{K^{2}},\frac{D}{K}\big)$) by $K$. It therefore suffices to prove the assertion for the case where $D$ and $\frac{N}{D}$ are co-prime. But if a prime $p$ divides $\frac{N}{D}$ then $\Delta(N,D)$ contains more than one cyclic subgroup of order $p$, while otherwise it does not. As this easily determines $D$ and $\frac{N}{D}$, this completes the proof of the proposition.
\end{proof}

\smallskip

We conclude with investigating what happens when $\mathbb{R}$ is replaced by other fields (which must be of characteristic 0 since $\mathbb{Z}$ is assumed to be included as a subring). Let $\overline{\mathbb{Q}}$ be an algebraic closure of $\mathbb{Q}$ (embedded in $\mathbb{C}$, say), and let $\mathbb{Q}(\sqrt{\mathbb{Q}})$ be the compositum of all the quadratic fields. In addition we denote, for any field $\mathbb{F}$, the subgroup of $GL_{2}(\mathbb{F})$ consisting of those matrices whose determinant lies in $\{\pm1\}$ by $GL_{2}^{\pm1}(\mathbb{F})$
\begin{prop}
\begin{enumerate}[$(i)$]
\item The normalizer of any group $\Gamma_{H}$ in $SL_{2}(\mathbb{R})$ coincides with the normalizer in $SL_{2}(\overline{\mathbb{Q}}\cap\mathbb{R})$, as well as with the normalizer in $SL_{2}\big(\mathbb{Q}(\sqrt{\mathbb{Q}})\cap\mathbb{R}\big)$.
\item The normalizers in $SL_{2}(\mathbb{C})$, $SL_{2}(\overline{\mathbb{Q}})$, and $SL_{2}\big(\mathbb{Q}(\sqrt{\mathbb{Q}})\big)$ also coincide. This common group is an extension of the group from part $(i)$ by an element squaring to $-I$, having a simple action on the former group.
\item The normalizer of $\Gamma_{H}$ in $GL_{2}^{\pm1}(\mathbb{R})$, as well as in $GL_{2}^{\pm1}(\overline{\mathbb{Q}}\cap\mathbb{R})$ and in $GL_{2}^{\pm1}\big(\mathbb{Q}(\sqrt{\mathbb{Q}})\cap\mathbb{R}\big)$, is the semi-direct product of the group from part $(i)$ with a cyclic group of order 2, the generator of which acts like the element from part $(ii)$.
\item The group $SAut\big(L(N,D)\big)$, without the $+$ restriction, is a semi-direct product involving $SAut^{+}\big(L(N,D)\big)$ as in part $(iii)$. Removing the determinant restriction, the groups $Aut\big(L(N,D)\big)$ and $Aut^{+}\big(L(N,D)\big)$ are obtained as the direct product of the $SAut$ (resp. $SAut^{+}$) group with the automorphism $-Id_{L(N,D)}$ of global inversion.
\end{enumerate} \label{fielddef}
\end{prop}

\begin{proof}
Part $(i)$ follows directly from Proposition \ref{ambgrp} and the fact that the formula from Definition \ref{Gamma0*sNdef} involves only square roots of non-negative rational numbers. For part $(ii)$ we observe that one place in the proof of Proposition \ref{ambgrp} where we have used the fact that our matrix has real entries is when we said that if $g=0$ then $e=h=\pm1$. Allowing complex coefficients (or coefficients from any other algebraically closed field of characteristic 0), we obtain also the possibility where $a=-d=\pm i$ (with $i=\sqrt{-1}$), yielding just real matrices multiplied by $\big(\begin{smallmatrix} i & 0 \\ 0 & -i\end{smallmatrix}\big)$. Since we end up in the algebraic extension $\mathbb{Q}(\sqrt{\mathbb{Q}})$ of $\mathbb{Q}$, we may ignore questions of compatibility between real entries and entries of the algebraically closed field over which we are working. The only other place in that proof where reality was used is where the number $t$ was assumed to be positive. As allowing $t$ to be negative is the same as multiplying a matrix from $SL_{2}(\mathbb{R})$ by $\big(\begin{smallmatrix} i & 0 \\ 0 & -i\end{smallmatrix}\big)$ as well, we find that the group in question contains the normalizer from above as a subgroup of index 2, with which $\big(\begin{smallmatrix} i & 0 \\ 0 & -i\end{smallmatrix}\big)$ generates the full group. Since this matrix squares to the non-trivial central element $-I$ of the real normalizer, and conjugation by which simply inverts the signs of the off-diagonal entries, this proves part $(ii)$. Part $(iii)$ is easily deduced by replacing $\big(\begin{smallmatrix} i & 0 \\ 0 & -i\end{smallmatrix}\big)$ by its real counterpart $\big(\begin{smallmatrix} 1 & 0 \\ 0 & -1\end{smallmatrix}\big)$, which has determinant $-1$ and order 2, and conjugation by which yields the same operation as in part $(ii)$. For part $(iv)$ we recall that conjugation by matrices of determinant $-1$ yields the operation of elements of $SO\big(M_{2}(\mathbb{R})_{0}\big)$ that do not preserve the orientations on the definite parts (check, e.g., the action of the matrix from the proof of part $(iii)$). Hence the first assertion follows from part $(iii)$, and the second one is a consequence of the fact that $-Id_{L(N,D)}$ is central, has determinant $-1$, and preserves the orientation on the 2-dimensional positive definite part. This proves the proposition.
\end{proof}
As a final remark, we mention that part $(iii)$ of Proposition \ref{fielddef} only allows determinants $\pm1$, restricting further just to determinant 1 in part $(ii)$ there, since we are only interested in groups whose center consists just of $\{\pm I\}$. This is done in order to avoid a trivial increase of all the normalizers.

\noindent\textsc{Einstein Institute of Mathematics, the Hebrew University of Jerusalem, Edmund Safra Campus, Jerusalem 91904, Israel}

\noindent E-mail address: zemels@math.huji.ac.il

\end{document}